\documentclass[final,leqno,onefignum,onetabnum]{siamltex1213}

\usepackage{showkeys}
\usepackage{amssymb,amsmath}
\usepackage{algorithm}
\usepackage{algorithmic}

\newtheorem{remark}{\bf  Remark}[section]
\newtheorem{assumption}{\bf  Assumption}

\newcommand{\bbE}{{\mathbb{E}}}

\newcommand{\bbR}{{\mathbb{R}}}
\newcommand{\bbN}{{\mathbb{N}}}
\newcommand{\bbJ}{{\mathbb{J}}}

\newcommand{\cB}{\mathcal{B}}
\newcommand{\cC}{\mathcal{C}}

\newcommand{\cF}{\mathcal{F}}

\newcommand{\cN}{\mathcal{N}}

\newcommand{\cR}{\mathcal{R}}

\newcommand{\cQ}{\mathcal{Q}}

\newcommand{\argmax}{\operatornamewithlimits{argmax}}
\newcommand{\argmin}{\operatornamewithlimits{argmin}}

\newcommand{\bskappa}{{\boldsymbol{\kappa}}}
\newcommand{\bsnull}{{\boldsymbol 0}}

\newcommand{\bsrho}{{\boldsymbol{\rho}}}

\newcommand{\bsnu}{{\boldsymbol{\nu}}}
\newcommand{\bsmu}{{\boldsymbol{\mu}}}

\newcommand{\bse}{{\boldsymbol{e}}}

\newcommand{\bsy}{{\boldsymbol{y}}}

\newcommand{\bsgamma}{{\boldsymbol{\gamma}}}

\newcommand{\beq}{\begin{equation}}
\newcommand{\eeq}{\end{equation}}
\newcommand{\ba}{\begin{array}}
\newcommand{\ea}{\end{array}}

\title{
Sparse polynomial approximation
for optimal control problems constrained by elliptic PDEs with lognormal random coefficients 
\thanks{This work is 
   supported by 
    DOE grant DE-SC0019393, AFOSR grant FA9550-17-1-0190
    and NSF grant ACI-1550593.} 
}

\author{Peng Chen \thanks{Oden Institute for Computational Engineering \&
    Sciences, The University of Texas at Austin, Austin, TX 78712
    (\email{peng@ices.utexas.edu}).}   
    \and Omar Ghattas \thanks{Oden Institute
    for Computational Engineering \& Sciences, Department of
    Mechanical Engineering, and Department of Geological Sciences, The
    University of Texas at Austin, Austin, TX 78712
    (\email{omar@ices.utexas.edu}).}  
    }
\begin{document}
\maketitle
\slugger{sisc}{xxxx}{xx}{x}{x--x}

\begin{abstract}
In this work, we consider optimal control problems constrained by elliptic partial differential equations (PDEs) with lognormal random coefficients, which are represented by a countably infinite-dimensional random parameter with i.i.d. normal distribution. We approximate the optimal solution by a suitable truncation of its Hermite polynomial chaos expansion, which is known as a sparse polynomial approximation. Based on the convergence analysis in \cite{BachmayrCohenDeVoreEtAl2017} for elliptic PDEs with lognormal random coefficients, we establish the dimension-independent convergence rate of the sparse polynomial approximation of the optimal solution. Moreover, we present a polynomial-based sparse quadrature for the approximation of the expectation of the optimal solution and prove its dimension-independent convergence rate based on the analysis in \cite{Chen2018}. Numerical experiments demonstrate that the convergence of the sparse quadrature error is independent of the active parameter dimensions and can be much faster than that of a Monte Carlo method.
\end{abstract}



\pagestyle{myheadings}
\thispagestyle{plain}
\markboth{Sparse polynomial approximation for optimal control with lognormal random coefficients}{P. Chen 
and O. Ghattas
}

\section{Introduction}

PDE-constrained optimal control or optimization problems arise in many areas of engineering and science. These problems can be generally formulated as minimization of a cost functional subject to the PDEs that model the behavior of the physical systems we seek to control. The control function can be either a distributed control defined in the physical domain, a boundary control that operates on its boundary, or a shape control that seeks to optimize the shape or geometry of the system. The cost functional often involves two terms. The first depends on the state of the system and reflects the control objective to be optimized. The second is a regularization or penalty term that reflects the regularity or cost of the control function. Theoretical analyses of the existence and uniqueness of the optimal solutions and the numerical approximation of these optimal control problems, as well as the design of efficient computational methods have been well studied during the last few decades, see the classical books \cite{Lions1971, GlowinskiLions95, Gunzburger03, HinzePinnauUlbrichEtAl2008, Troeltzsch2010, BorziSchulz2011} and references therein.  

In many optimal control problems, the PDEs that govern system behavior are characterizing by uncertain fields representing, for example, initial or boundary conditions, heterogeneous coefficients, or geometry. In such cases, it is important to account for this uncertainty  for reliability and robustness of the the optimal control.
In recent years, optimal control  problems under uncertainty have received increased attention \cite{BorziSchulzSchillingsEtAl2010, SchillingsSchmidtSchulz2011, HouLeeManouzi2011,GunzburgerLeeLee11,RosseelWells2012,
KouriHeinkenschloosVanBloemenWaanders2012,TieslerKirbyXiuEtAl2012,ChenQuarteroniRozza2013,LassilaManzoniQuarteroniEtAl2013, ChenQuarteroni2014, KunothSchwab2013, NgWillcox2014, ChenQuarteroniRozza2016, KunothSchwab2016, KouriSurowiec2016, BennerOnwuntaStoll2016, AlexanderianPetraStadlerEtAl2017, AliUllmannHinze2017}. 
Topics include the mathematical representation of the uncertainties, computational methods to solve the stochastic optimal control problems, probability or risk measures for the control objective, and stochastic or deterministic formulations of the control functions. We elaborate more on the former two aspects relevant to this work.

First, proper representation of the uncertainties, especially those with spatially varying randomness, is of primary importance not only for the analysis of mathematical properties of optimal control problems but also for the development of computational methods to solve these problems. Most work has considered a small or finite number of uniformly distributed random variables to represent the uncertainties for tractability \cite{BorziSchulzSchillingsEtAl2010, SchillingsSchmidtSchulz2011, HouLeeManouzi2011,GunzburgerLeeLee11,RosseelWells2012,
KouriHeinkenschloosVanBloemenWaanders2012,TieslerKirbyXiuEtAl2012,ChenQuarteroniRozza2013,LassilaManzoniQuarteroniEtAl2013, ChenQuarteroni2014, NgWillcox2014, ChenQuarteroniRozza2016, KouriSurowiec2016, BennerOnwuntaStoll2016}. However, in many applications, such as subsurface flow where the permeability coefficient is most often modeled as an infinite-dimensional lognormal random field \cite{HoangSchwab14, AlexanderianPetraStadlerEtAl2017, GrahamKuoNicholsEtAl15, BachmayrCohenMigliorati2016, BachmayrCohenDeVoreEtAl2017}, we need to consider more general representations of the uncertainties, e.g., in terms of countably infinite number of random variables with suitable probability distributions. 

Second,  accurate and efficient numerical approximation methods play a key role in solving PDE-constrained optimal control problems under uncertainty, especially in the case of high-dimensional uncertainty represented by a large number of random variables. Monte Carlo methods are widely used because of their straightforward and noninstrusive implementation, as well as their dimension-independent convergence. However, their $O(N^{-1/2})$ convergence with $N$ samples often results in a need for a large number of samples to achieve a certain desired accuracy. Various improvements such as quasi or multilevel Monte Carlo methods can effectively reduce the total computational cost \cite{AliUllmannHinze2017}, even though the convergence of the methods remains slow. Another important class of computational methods are the polynomial chaos-based stochastic Galerkin \cite{HouLeeManouzi2011, RosseelWells2012, LeeLee2013, KunothSchwab2016} and stochastic collocation \cite{Borzi2010, KouriHeinkenschloosVanBloemenWaanders2012, TieslerKirbyXiuEtAl2012, ChenQuarteroniRozza2013, Kouri2014} methods, which achieve fast convergence for smooth problems for a relatively moderate number of dimensions. To deal with the need to solve the governing PDEs numerous times, model reduction methods such as reduced basis or proper othogonal decomposition methods have been developed \cite{BorziWinckel2011, GunzburgerMing2011, ChenQuarteroni2014, ChenQuarteroniRozza2016, PeherstorferWillcoxGunzburger18}. Low-rank tensor methods have also been proposed to solve such problems in \cite{BennerOnwuntaStoll2016, BennerDolgovOnwuntaEtAl2016, GarreisUlbrich2017}. Recently, Taylor approximation of the objective function with respect to the random field and variance reduction have been developed \cite{AlexanderianPetraStadlerEtAl2017, ChenVillaGhattas17}, and have been demonstrated to be scalable for high-dimensional control problems with moderate uncertainty. 


In this work, we consider optimal control problems constrained by elliptic PDEs with lognormal diffusion coefficients that are often used to represent signed or positive random fields. The logarithm of such coefficients can be represented by an expansion on a countably infinite number of basis functions, where the randomness is described by the random coefficients obeying i.i.d.\ standard normal distribution. This representation naturally arises from a Karhunen--Lo\`eve expansion or wavelet expansion of Gaussian random fields with the second moment \cite{HoangSchwab14, GrahamKuoNicholsEtAl15, BachmayrCohenMigliorati2016, BachmayrCohenDeVoreEtAl2017}. To characterize the smoothness of the random field, we assume that a weighted sum of the basis functions is finite with the sequence of weights decaying in a certain algebraic rate, as studied in \cite{BachmayrCohenDeVoreEtAl2017}. We then consider distributed optimal control problems constrained by elliptic PDEs with lognormal random coefficients, where the specific cost functional we consider consists of the expectation of the deviation of the PDE state from a desired state, and a regularization term on the control function that is assume to depend on the realization of the random field. With such a asetting, the optimal solution of the control problem can be obtained by solving a first order optimality system. By a reduced formulation of the optimality system, we propose to approximate its solution by sparse polynomials, or tensorized polynomials in suitable sparse index sets. More specifically, we employ the Hermite polynomial chaos expansion of the optimal solution, and truncate the expansion with $N$ terms that correspond to the $N$ largest Hermite coefficients. We also present a polynomial-based sparse quadrature, which is a sum of tensorized univariate quadrature in suitable sparse index sets, to compute statistical moments of the optimal solution, such as its expectation. For practical construction of the sparse quadrature, we present an adaptive algorithm with both a-priori and a-posteriori error indicators. The main contribution of this work is the analysis and demonstration of the convergence property of the sparse polynomial approximation and the sparse quadrature for the solution of the optimal control problems under the lognormal uncertainty. In particular, based on the analysis for elliptic PDEs with lognormal random coefficients in the recent work \cite{BachmayrCohenDeVoreEtAl2017}, we establish the dimension-independent convergence rate $O(N^{-s})$ of the sparse Hermite polynomial approximation, where $s > 0$ depends only on parametrization of the lognormal random field, and not on the dimension of the parameter space, thus overcoming the curse of dimensionality and achieving fast convergence for large $s$. We remark that a similar convergence rate is obtained in \cite{KunothSchwab2013} by truncation of Legendre polynomial chaos expansion for a sequence of uniformly distributed random variables. Moreover, for sparse quadrature in approximating the expectation of the optimal solution, we prove that its error also converges with dimension-independent convergence rate $O(N^{-s})$ with a different $s > 0$. We demonstrate the dimension-independent convergence property of the sparse quadrature by a 1025-dimensional stochastic optimal control problem.

The rest of the paper is organized as follows. In Section \ref{sec:setting}, we first present an elliptic PDE with lognormal random field and its parametrization. Then the stochastic optimal control problem is presented, for which we formulate a reduced optimality system and show the existence of a unique solution of this system as well as the finite moments of the optimal solution. In Section \ref{sec:sparsePolynormial}, the Hermite polynomial expansion and truncation are introduced to formulate the sparse polynomial approximation. The dimension-independent convergence rate of this approximation is established based on the $\ell^p$-summability of the Hermite coefficients. Based on sparse polynomial approximation, we present a sparse quadrature with an adaptive construction algorithm using both a-priori and a-posteriori error indicators
in Section \ref{sec:sparseQuadrature}. We also prove its dimension-independent convergence property under assumptions on the exactness and boundedness of the univariate quadrature. We perform numerical experiments to demonstrate the convergence property of the sparse quadrature constructed by both a-priori and a-posteriori error indicators in Section \ref{sec:numerics}. Finally, we draw conclusions and mention some further research topics in Section \ref{sec:conclusion}.

\section{Problem setting}
\label{sec:setting}
In this section, we present optimal control problems constrained by elliptic PDEs with lognormal random coefficients. For the coefficients, we consider an explicit infinite-dimensional parametrization. For the elliptic PDE-constrained optimal control problems, we present its optimality system, well-posedness, and finite moments of the optimal solution. 

\subsection{Elliptic PDEs with lognormal random coefficients}
Let $D \in \bbR^d$ ($d = 1, 2, 3$) denote an open and bounded physical domain with Lipschitz boundary $\partial D$. We consider the following elliptic PDE with suitable boundary conditions:
\beq\label{eq:ellipticPDE}
- \nabla \cdot (e^\kappa \nabla u) = f  \quad \text{ in } D,
\eeq
where $u$ is the state variable, $f$ is the source term, $e^\kappa$ is the diffusion coefficient with $\kappa$ a Gaussian random field. We refer to $e^\kappa$ as a lognormal random coefficient. We assume that the random field $\kappa$ can be represented by 
\beq\label{eq:kappa}
\kappa(x,\bsy) = \sum_{j \geq 1}y_j \kappa_j(x),
\eeq 
where $(\kappa_j)_{j \geq 1}$ in \eqref{eq:kappa} is a sequence of functions in $L^{\infty}(D)$, $\bsy = (y_j)_{j \geq 1}$ is a sequence of i.i.d.\ standard normally distributed random variables, which can be viewed as a parameter vector with the Gaussian probability measure $\gamma = \cN(0, 1)$ for each element. 
We denote the unbounded domain for the parameter vector $\bsy$ as 
\beq
Y := \bbR^\bbN,
\eeq
and consider the product measure space 
\beq
(Y, \cB(Y), \bsgamma) = (\bbR^\bbN, \cB(\bbR^\bbN),  \bsgamma),
\eeq
where $\cB(\bbR^\bbN)$ denotes the $\sigma$-algebra generated by the Borel cylinders and $\bsgamma = \prod_{j\geq 1} \gamma_j$ denotes the tensorized Gaussian probability measure. By $L^2_\bsgamma(Y)$ we denote a Hilbert space of square integrable functions with respect to the measure $\bsgamma$. The representation \eqref{eq:kappa} naturally arises from Karhunen--Lo\`eve (KL) expansion of the Gaussian random field with measure $\cN(0, \cC)$, for which 
\beq
\kappa_j := \sqrt{\lambda_j} \psi_j,
\eeq
where $(\lambda_j, \psi_j)_{j\geq 1}$ denote the eigenpairs of the covariance operator $\cC$. Alternatively, $(\kappa_j)_{j\geq 1}$ can be constructed using certain wavelet basis functions that have local support \cite{BachmayrCohenMigliorati2016}, 
\beq
\kappa_j := \psi_{\lambda},
\eeq
where $\lambda \in \triangle$ denotes a space-scale index as in \cite{BachmayrCohenDeVoreEtAl2017}.
 The smoothness of the Gaussian random field is related to smoothness of the covariance $\cC \in C^\beta(D\times D)$ for some $\beta > 0$ \cite{Charrier2012}, or related to the decay of the basis functions $||\psi_j||_{L^\infty}$ as $j \to \infty$ \cite{DashtiStuart17}. To characterize the smoothness property of $\kappa$, we make the following assumption, which covers both the KL-type and the wavelet-type representations as in \cite{BachmayrCohenDeVoreEtAl2017}. 

\begin{assumption}\label{ass:kappa}
Let $0 < p < 2$ and $q = \frac{2p}{2-p}$. Assume there exists a positive sequence $\bsrho = (\rho_j)_{j\geq 1}$ such that $(\rho_j^{-1})_{j\geq 1} \in \ell^q(\bbN)$ and such that
\beq\label{eq:rhokappa}
\sup_{x\in D} \sum_{j\geq 1} \rho_j |\kappa_j(x)| < \infty.
\eeq
\end{assumption}

\begin{proposition}\label{prop:finiteMoments}
\cite[Theorem 2.1, Remark 2.2]{BachmayrCohenDeVoreEtAl2017} 
Under Assumption \ref{ass:kappa}, the map $\bsy \to \kappa(\bsy)$ is measurable, and for any $0 \leq n < \infty $, there holds 
\beq
\bbE[\exp(n ||\kappa ||_{L^\infty(D)})] < \infty.
\eeq 
\end{proposition}

By $L^2(D)$ we denote the space of square integrable functions. Let $H^1(D) := \{v \in L^2(D): |\nabla v| \in L^2(D)\}$, and $V = H^1_0(D) := \{v \in H^1(D): v|_{\partial D} = 0\}$.  Then the weak formulation of the elliptic PDE \eqref{eq:ellipticPDE} with homogeneous Dirichlet boundary condition reads: given $\bsy \in Y$ and $f \in L^2(D)$, find $u(\bsy) \in V$ such that 
\beq\label{eq:ellipticPDEweak}
a(u(\bsy), v; \kappa(\bsy)) = (f, v) \quad \forall v \in V,
\eeq
where the bilinear form $a(\cdot, \cdot; \kappa)$ is given by 
\beq
a(w,v; \kappa(\bsy)) = \int_{D} e^{\kappa(\bsy)} \nabla w \cdot \nabla v dx, \quad \forall w, v \in V,
\eeq
and $(\cdot, \cdot)$ denotes the inner product in $L^2(D)$.
Well-posedness and finite moments properties are obtained for the elliptic problem \eqref{eq:ellipticPDEweak} in \cite{BachmayrCohenDeVoreEtAl2017}, as stated in the following theorem.

\begin{theorem}
Under Assumption \ref{ass:kappa}, for any $\bsy \in Y$, there exists a unique solution $u(\bsy)$ of problem \eqref{eq:ellipticPDEweak}. Moreover, for any $0 \leq k < \infty$, there holds
\beq
\bbE[||u||_{V}^k] = \int_Y ||u(\bsy)||_{V}^k d\bsgamma(\bsy)< \infty.
\eeq
\end{theorem}

\subsection{Stochastic optimal control problems} We consider an elliptic PDE with a distributed control function and homogeneous Dirichlet boundary condition as
\beq\label{eq:ellipticControl}
- \nabla \cdot (e^\kappa \nabla u ) = f + z \quad \text{ in } D,
\eeq
where $z$ denotes the control function in $L^2(D)$. Moreover, we assume that $z$ depends on the realization of $\bsy \in Y$ as in \cite{GunzburgerLeeLee11, HouLeeManouzi2011, KouriHeinkenschloosVanBloemenWaanders2012, ChenQuarteroni2014, KunothSchwab2013, ChenQuarteroniRozza2013, KunothSchwab2016, AliUllmannHinze2017}. Then the weak formulation of \eqref{eq:ellipticControl} is given by: given $\bsy \in Y$, find $u(\bsy) \in V$ such that 
\beq
a(u(\bsy), v; \kappa(\bsy)) = (f+z(\bsy), v) \quad \forall v \in V,
\eeq
where $(\cdot, \cdot)$ denotes the inner product in $L^2(D)$. 
We consider the cost functional 
\beq\label{eq:costFunctional}
J(u, z) = \frac{1}{2}\bbE[||u - u_d||_{L^2(D)}^2] + \frac{\beta}{2}\bbE[||z||_{L^2(D)}^2],
\eeq 
where the first term represents a tracking-type control objective with $u_d$ as desired state, $\beta > 0$ is a weighting parameter. Then the stochastic optimal control problem constrained by the elliptic PDE problem \eqref{eq:ellipticControl} is formulated as:
\beq\label{eq:stochasticOC}
\min_{z \in Z} J(u(z), z) \text{ such that } \eqref{eq:ellipticControl} \text{ is satisfied for every } \bsy \in Y,
\eeq 
where $Z$ denotes admissible control function space set as $Z = L_\bsgamma^2(Y, L^2(D))$, a Bochner space equivalent to the tensor product function space $L_\bsgamma^2(Y) \otimes L^2(D)$ of square integrable functions in both stochastic and physical domains.

\begin{remark}
We remark that problem \eqref{eq:stochasticOC} with a stochastic distributed control function and the cost functional \eqref{eq:costFunctional} with expectation as risk measure is a particular (maybe simplest) type of stochastic optimal control problems as considered in \cite{GunzburgerLeeLee11, HouLeeManouzi2011, KouriHeinkenschloosVanBloemenWaanders2012, ChenQuarteroni2014, KunothSchwab2013, ChenQuarteroniRozza2013, KunothSchwab2016, AliUllmannHinze2017}. More general stochastic optimal control problems include those with deterministic control functions,  boundary or shape control functions, and more general PDE models, as well as other risk measures in the cost functional \cite{SchillingsSchmidtSchulz2011, TieslerKirbyXiuEtAl2012, ChenVillaGhattas17, KouriSurowiec2016, LeeGunzburger2017}.
\end{remark}

%
%

\subsection{Optimality system, well-posedness, and finite moments}

To derive the first order optimality system for the linear-quadratic stochastic optimal control problem \eqref{eq:stochasticOC}, with linear PDE constraint and quadratic cost functional, we use a Lagrange multiplier approach. First, we form the Lagrangian 
\beq
L(u,z,v) = J(u,z) + \bbE[(f+z,v) - a(u,v;\kappa)],
\eeq
with the adjoint variable $v \in L^2_\bsgamma(Y, V)$. Then by setting the first order variation of the Lagrangian with respect to the adjoint $v$, the state $u$, and the control $z$ to be zero, we obtain the optimality system: find $(u,v,z) \in L^2_\bsgamma(Y, V)\times L^2_\bsgamma(Y, V) \times L^2_\bsgamma(Y, L^2(D)$ such that  
\beq\label{eq:expOC}
\left\{
\begin{array}{lll}
\bbE[a(u,\tilde{u};\kappa)] &=& \bbE[ (f+z, \tilde{u})] \quad \forall \tilde{u} \in L^2_\bsgamma(Y, V),\\
\bbE[a(\tilde{v}, v; \kappa)] &=& \bbE[(\tilde{v}, u-u_d)] \quad \forall \tilde{v} \in L^2_\bsgamma(Y, V),\\
\bbE[\beta(\tilde{z}, z)] &=& \bbE[-(\tilde{z}, v)] \quad \forall \tilde{z} \in L^2_\bsgamma(Y, L^2(D)).
\end{array}
\right.
\eeq 
Then the classical theory \cite{Lions1971, HinzePinnauUlbrichEtAl2008} states that there exists a unique solution of the linear-quadratic optimal control problem \eqref{eq:stochasticOC}, which is the solution of the optimality system \eqref{eq:expOC}.
Eliminating $z$ from the first equation by the third equation, we obtain the reduced optimality problem: find $(u,v) \in L^2_\bsgamma(Y, V) \times L^2_\bsgamma(Y, V)$ such that 
\beq\label{eq:expectationOC}
\left\{
\begin{array}{lll}
\bbE[a(u,\tilde{u};\kappa) + \frac{1}{\beta} (\tilde{u}, v)]&=& \bbE[(f, \tilde{u})] \quad \forall \tilde{u} \in L^2_\bsgamma(Y, V),\\
\bbE[a(\tilde{v}, v; \kappa) - (\tilde{v}, u)] &=& \bbE[(\tilde{v}, -u_d)] \quad \forall \tilde{v} \in L^2_\bsgamma(Y, V).
\end{array}
\right.
\eeq 
Multiplying the first equation by $\beta$ and adding it to the second equation, we obtain
\beq\label{eq:expectationOCOne}
\bbE[\beta a(u, \tilde{u}; \kappa) + a(\tilde{v}, v;\kappa) + (\tilde{u}, v) - (\tilde{v}, u)] = \bbE[\beta(f,\tilde{u}) + (\tilde{v}, -u_d)],
\eeq 
holding for any $(\tilde{u}, \tilde{v}) \in L^2_\bsgamma(Y, V) \times L^2_\bsgamma(Y, V)$. 

For each $\bsy \in Y$,
we denote
\beq\label{eq:wuv}
w(\bsy) = (u(\bsy), v(\bsy)) \in W, \text{ where } W = V \times V,
\eeq
associated with the norm 
\beq\label{eq:Wnorm}
||w(\bsy)||_W = \left(\beta|| u(\bsy) ||^2_V + || v(\bsy) ||^2_V\right)^{1/2},
\eeq
where the $||\cdot||_V$-norm is specified as 
\beq\label{eq:Vnorm}
 ||v||_V = || \,|\nabla v | \, ||_{L^2(D)}, \quad \forall v \in V.
\eeq
We consider the $\bsy$-pointwise reduced optimality problem corresponding to \eqref{eq:expectationOCOne} in the weak form: given $\bsy \in Y$, find $w(\bsy) \in W$ such that
\beq\label{eq:pointwiseOC}
A(w(\bsy), \tilde{w}; \kappa(\bsy)) + B(w(\bsy), \tilde{w}) = (F, \tilde{w}) \quad \forall \tilde{w} \in W,
\eeq
where we denote $\tilde{w} = (\tilde{u}, \tilde{v})\in W$, and
\beq
A(w(\bsy), \tilde{w}; \kappa(\bsy)) = \beta a(u(\bsy), \tilde{u}; \kappa(\bsy)) + a(v(\bsy), \tilde{v}; \kappa(\bsy)),
\eeq
and
\beq
B(w(\bsy), \tilde{w}) = (v(\bsy),\tilde{u}) - (u(\bsy), \tilde{v}), \text{ and } (F, \tilde{w}) = \beta(f, \tilde{u}) - (u_d, \tilde{v}).
\eeq

\begin{theorem}\label{thm:yOC}
For any given $\bsy \in Y$ such that $||\kappa(\bsy)||_{L^\infty(D)} < \infty $, there exists a unique solution $w(\bsy) \in W$ of the problem \eqref{eq:pointwiseOC}, such that there holds the a-priori estimate
\beq\label{eq:estimateW}
||w(\bsy)||_W \leq C_P\left(\sqrt{\beta}||f||_{L^2(D)} + ||u_d||_{L^2(D)}\right) \exp(||\kappa(\bsy)||_{L^\infty(D)}),
\eeq
where $C_P$ is the Poincar\'e constant \cite[Property 2.4]{Quarteroni2013}.  Moreover, under Assumption \ref{ass:kappa}, for any $0 \leq n < \infty$, there holds 
\beq\label{eq:finiteMoments}
\bbE[||w||_W^n] = \int_Y ||w(\bsy)||_W^n d\bsgamma(\bsy)< \infty.
\eeq
\end{theorem}
\begin{proof}
It is straightforward to see that the bilinear form $A(\cdot, \cdot;\kappa(\bsy)) + B(\cdot, \cdot)$ is continuous. Moreover, it is coercive as $B(w,w) = 0$, $\forall w \in W$, and 
\beq\label{eq:coerciveA}
A(w, w; \kappa(\bsy)) \geq \exp(-||\kappa(\bsy)||_{L^\infty(D)}) ||w||_W^2, \quad \forall w \in W.
\eeq
Furthermore, the linear form $(F, \cdot)$ is bounded as 
\beq
\begin{split}
(F, w) &\leq \sqrt{\beta} ||f||_{L^2(D)} \sqrt{\beta}||u||_{L^2(D)} + ||u_d||_{L^2(D)} ||v||_{L^2(D)}\\
& \leq C_P \sqrt{\beta} ||f||_{L^2(D)} \sqrt{\beta}||\, |\nabla u| \,||_{L^2(D)} + C_P ||u_d||_{L^2(D)} || \, |\nabla v| \, ||_{L^2(D)} \\
& \leq C_P (\sqrt{\beta} ||f||_{L^2(D)}  + ||u_d||_{L^2(D)}) ||w||_W, \quad \forall w \in W,
\end{split}
\eeq
where we used the Poincar\'e inequality \cite[Property 2.4]{Quarteroni2013} in the second inequality. This bound, together with \eqref{eq:coerciveA}, implies the unique solution by the Lax--Milgram theorem \cite[Lemma 3.1]{Quarteroni2013}, which satisfies the a-priori estimate \eqref{eq:estimateW}.

The finite moments \eqref{eq:finiteMoments} is a result of \eqref{eq:estimateW} and Proposition \ref{prop:finiteMoments}.
\end{proof}

\begin{remark}
With such a $\bsy$-pointwise formulation \eqref{eq:pointwiseOC}, from the unique solution of \eqref{eq:pointwiseOC} for every $\bsy \in Y$ and the finite moments \eqref{eq:finiteMoments} for $n = 2$, we have that there exists a unique solution of \eqref{eq:expectationOCOne}, which is the unique solution of \eqref{eq:expectationOC},  \eqref{eq:expOC}, and \eqref{eq:stochasticOC} with the optimal control function given by $z = -v/\beta$.
\end{remark}

%

\section{Sparse polynomial approximation}
\label{sec:sparsePolynormial}
In this section, we present the Hermite polynomial chaos expansion of the solution of the reduced optimality system \ref{eq:expectationOCOne}. Based on the Hermite expansion, we define a sparse polynomial approximation and prove its dimension-independent convergence rate based on the $\ell^p$-summability of the Hermite coefficients. 
\subsection{Hermite polynomial chaos}
Let $\bsnu = (\nu_1, \nu_2, \dots,)$ denote a multi-index with $\nu_j \in \bbN \cup \{0\}$ for every $j \geq 1$. Let $|\bsnu| = \sum_{j\geq 1} \nu_j$ and $||\bsnu||_{\ell^\infty} = \max_{j\geq 1}\nu_j$.
Let $\cF$ denote a multi-index set with finitely supported indices, i.e., 
\beq
\cF = \{\bsnu \in \bbN^\bbN: |\bsnu| < \infty\}. 
\eeq
Let $(H_{n})_{n \geq 0}$ denote a sequence of orthonormal Hermite polynomials, and $H_{\bsnu}$ denote a tensorized Hermite polynomial given by 
\beq
H_{\bsnu}(\bsy) = \prod_{j\geq 1} H_{\nu_j}(y_j). 
\eeq 
Then $(H_\bsnu)_{\bsnu \in \cF}$ form a complete orthonormal basis for the Hilbert space $L^2_\bsgamma(Y)$. By Theorem \ref{thm:yOC}, we have that the solution of problem \eqref{eq:expectationOCOne} satisfy  $u, v \in L^2_\bsgamma(Y, V)$. Therefore, $w = (u,v)$ admits the Hermite polynomial chase expansion
\beq\label{eq:hermite}
w(\bsy) = \sum_{\bsnu \in \cF} w_\bsnu H_\bsnu(\bsy),
\eeq
where the Hermite coefficients $w_\bsnu = (u_\bsnu, v_\bsnu)$ are given by
\beq
u_\bsnu = \int_Y u(\bsy) H_\bsnu(\bsy) d\bsgamma(\bsy) \text{ and } v_\bsnu = \int_Y v(\bsy) H_\bsnu(\bsy) d\bsgamma(\bsy), \quad \forall \bsnu \in \cF.
\eeq
Moreover, by Parseval's identity, we have
\beq
||u||_{L^2_\bsgamma(Y, V)}^2 = \sum_{\bsnu \in \cF} ||u_\bsnu||_{V}^2 \text{ and } ||v||_{L^2_\bsgamma(Y, V)}^2 = \sum_{\bsnu \in \cF} ||v_\bsnu||_{V}^2.
\eeq 
Therefore, by definition of the $||\cdot||_W$-norm in \eqref{eq:Wnorm}, we have 
\beq
||w||_{L^2_\bsgamma(Y, W)}^2 = \sum_{\bsnu \in \cF} \left( \beta ||u_\bsnu||_{V}^2 +  ||v_\bsnu||_{V}^2\right) = \sum_{\bsnu \in \cF}  ||w_\bsnu||_{W}^2
\eeq
\subsection{Sparse polynomial approximation}
Let $\Lambda \subset \cF$ denote a multi-index set with cardinality $|\Lambda| < \infty$, we define a sparse polynomial approximation of $w$ as 
\beq
S_{\Lambda} w(\bsy) = \sum_{\bsnu \in \Lambda} w_\bsnu H_\bsnu(\bsy).
\eeq
Then the approximation error can be bounded as
\beq\label{eq:hPCerror}
||w - S_\Lambda w||_{L^2_\bsgamma(Y, W)}^2 = \Big|\Big| \sum_{\bsnu \in \cF \setminus \Lambda} w_\bsnu H_\bsnu \Big|\Big|_{L^2_\bsgamma(Y, W)}^2 \leq \sum_{\bsnu \in \cF \setminus \Lambda } ||w_\bsnu||_{W}^2.
\eeq
It is evident that the sparse polynomial approximation $S_\Lambda$, with $\Lambda$ taken such that the Hermite coefficients $||w_\bsnu||_{W}$ for $\bsnu \in \Lambda$ are the $|\Lambda|$ largest among all $\bsnu \in \cF$, is the optimal approximation in ${L^2_\bsgamma(Y, W)}$-norm with $|\Lambda|$ Hermite polynomials. It therefore becomes the task of quantifying the decay rate of the residual of the coefficients \eqref{eq:hPCerror} in order to obtain the convergence rate of the sparse polynomial approximation error. To this end, we first state the dimension-independent convergence rate of the sparse polynomial approximation in the following theorem. 

\begin{theorem}\label{thm:bestN}
Under Assumption \ref{ass:kappa}, there exists a sequence of multi-index set $(\Lambda_N)_{N \geq 1} \in \cF$ with $|\Lambda_N| = N$, such that 
\beq
||w - S_{\Lambda_N} w||_{L^2_\bsgamma(Y, W)} \leq C N^{-s}, \quad s = \frac{1}{p} - \frac{1}{2},
\eeq
where the constant $C > 0$ is independent of $N$.
\end{theorem}
\begin{proof}
Note that the sparse polynomial approximation error is bounded by \eqref{eq:hPCerror}. By Stechkin's lemma \cite{CohenDevoreSchwab2011}, taking $\Lambda_N$ with the $N$ elements $\bsnu \in \Lambda_N$ corresponding to the $N$ largest $||w_\bsnu||_W$ among all $\bsnu \in \cF$, there holds
\beq
\left(\sum_{\bsnu \in \cF \setminus \Lambda_N } ||w_\bsnu||_{W}^2\right)^{1/2} \leq \left(\sum_{\bsnu \in \cF} ||w_\bsnu||_{W}^p \right)^{1/p} N^{-s}, \quad s = \frac{1}{p} - \frac{1}{2},
\eeq
which concludes with $C = ||(||w_\bsnu||_{W})_{\bsnu \in \cF}||_{\ell^p(\cF)}$ if the Hermite coefficients $(||w_\bsnu||_{W})_{\bsnu \in \cF} \in \ell^p(\cF)$ under Assumption \ref{ass:kappa}. This is the result of Theorem \ref{thm:lp}.
\end{proof}

\subsection{$\ell^p(\cF)$-summability}
\label{sec:sum}
In this section, we study the $\ell^p(\cF)$-summability of the Hermite coefficients of the solution of the  reduced optimality problem \eqref{eq:pointwiseOC}, 
using similar arguments in \cite{BachmayrCohenDeVoreEtAl2017} for the elliptic PDE \eqref{eq:ellipticPDEweak} with a lognormal random coefficient. In particular, we need to bound the partial derivatives of the optimal solution and their weighted integrals.

By $\partial^\bsnu w$ we denote the $\bsnu$-th order partial derivative defined as 
\beq
\partial^\bsnu w := \Bigg( \prod_{j\geq 1} \partial^{\nu_j}_{y_j}\Bigg) w, \quad \forall \bsnu \in \cF.
\eeq 
We use the combinatorial notation 
\beq
{{\bsnu}\choose{\bsmu}} := 
{{\bsnu}\choose{\bsmu}}
= 
\prod_{j\geq 1}
{{\nu_j}\choose{\mu_j}}, \quad \bsnu, \bsmu \in \cF,
\eeq
with the convention 
\beq
{{n}\choose{m}}
:= 0, \quad \text{ if } \quad m > n.
\eeq
Let $\bskappa := (\kappa_j)_{j\geq 1}$ denote the sequence of basis functions of \eqref{eq:kappa}, we denote 
\beq
\bskappa^\bsnu := \prod_{j\geq 1} \kappa_j^{\nu_j}, \quad \forall \bsnu \in \cF. 
\eeq 
Moreover, for two indices $\bsnu, \bsmu \in \cF$, by $ \bsnu \preceq \bsmu$ we mean $\nu_j \leq \mu_j$ for all $j\geq 1$. We define the multi-index set $E_\bsmu$ for $\bsmu \in \cF$ as 
\beq
E_\bsmu = \{ \bsnu \in \cF: \bsnu \preceq \bsmu \text{ and } \bsnu \neq \bsmu \}, \quad \forall \bsmu \in \cF. 
\eeq

\begin{lemma}
\label{lem:dydwbound}
For any $\bsy \in Y$ such that $||\kappa(\bsy)||_{L^\infty(D)} < \infty $, for any $\bsmu \in \cF$ and $\bsmu \neq \bsnull$, there exists a unique partial derivative $\partial^\bsmu w(\bsy) \in W$ such that
\beq\label{eq:dydw}
\begin{split}
& A(\partial^\bsmu w(\bsy), \tilde{w}; \kappa(\bsy)) + B(\partial^\bsmu w(\bsy), \tilde{w}) \\
&= -\sum_{\bsnu \in E_\bsmu} 
{{\bsmu}\choose{\bsnu}}
\int_D \bskappa^{\bsmu-\bsnu} e^{\kappa(\bsy)} \nabla \partial^\bsnu w(\bsy) \cdot \nabla \tilde{w}
dx, \quad \forall \tilde{w} \in W,
\end{split}
\eeq
where we denote  
\beq
\nabla \partial^\bsnu w(\bsy) \cdot \nabla \tilde{w} := \beta \nabla \partial^\bsnu u(\bsy) \cdot \nabla \tilde{u} + \nabla \partial^\bsnu v(\bsy) \cdot \nabla \tilde{v}.
\eeq
Moreover, there holds the bounds 
\beq\label{eq:dydwbound}
||\partial^\bsmu w(\bsy)||_W \leq C_\bsmu \exp\left(2|\bsmu| \, ||\kappa(\bsy)||_{L^\infty(D)}\right), \quad \forall \bsmu \in \cF, 
\eeq
where the constant $C_\bsmu$ depends on $\bsmu$, $f$, $u_d$, and $(\kappa_j)_{j\geq 1}$, but not on $\bsy$. 
\end{lemma}
\begin{proof}
We proceed using the argument as in \cite[Lemma 3.1]{BachmayrCohenDeVoreEtAl2017} and \cite[Theorem 4.2]{CohenDevoreSchwab2011}. We first consider $\bsmu = \bse_j$ for some $j \geq 1$, where $(\bse_j)_i = \delta_{ij}$ denote the Kronecker sequence. Given any $\bsy \in Y$ such that $||\kappa(\bsy)||_{L^\infty(D)} < \infty $,  let $w_h(\bsy)$ denote 
\beq
w_h(\bsy) = \frac{w(\bsy+h \bse_j) - w(\bsy)}{h}.
\eeq
Subtracting \eqref{eq:pointwiseOC} at $\bsy+h\bse_j$ from it at $\bsy$ and dividing by $h$,  we obtain 
\beq\label{eq:Ah}
A(w_h(\bsy), \tilde{w}; \kappa(\bsy)) + B(w_h(\bsy), \tilde{w}) = - \int_D \frac{e^{h\kappa_j}-1}{h} e^{\kappa(\bsy)} \nabla w(\bsy+h\bse_j) \cdot \nabla \tilde{w} dx.
\eeq
Taking the limit $h \to 0$, we have 
\beq
\lim_{h \to 0} \int_D \frac{e^{h\kappa_j}-1}{h} e^{\kappa(\bsy)} \nabla w(\bsy+h\bse_j) \cdot \nabla \tilde{w} dx = \int_D \kappa_j e^{\kappa(\bsy)} \nabla w(\bsy) \cdot \nabla \tilde{w} dx.
\eeq
Therefore, taking limit $h \to 0$ in \eqref{eq:Ah} concludes \eqref{eq:dydw} for $\bsmu = \bse_j$, where we denote 
\beq
\lim_{h \to 0} w_h(\bsy) = \partial^{\bse_j} w(\bsy).
\eeq 
A recursive application of this argument concludes \eqref{eq:dydw} for any other $\bsmu \in \cF$ and $\bsmu \neq \bsnull$. To show the bound \eqref{eq:dydwbound}, we first note that it holds for $\bsmu = \bsnull$ by Theorem \ref{thm:yOC} with constant $C_\bsmu = C_P\left(\sqrt{\beta}||f||_{L^2(D)} + ||u_d||_{L^2(D)}\right)$. For any other $\bsmu \in \cF$, suppose that \eqref{eq:dydwbound} holds for any $\bsnu \in E_\bsmu$. We take the test function $\tilde{w} = \partial^\bsmu w(\bsy)$, which leads to $B(\partial^\bsmu w(\bsy), \partial^\bsmu w(\bsy)) = 0$, and 
\beq\label{eq:Absmu}
A(\partial^\bsmu w(\bsy), \partial^\bsmu w(\bsy); \kappa(\bsy)) \geq \exp\left(-||\kappa(\bsy)||_{L^\infty(D)}\right) ||\partial^\bsmu w(\bsy)||_W^2.
\eeq
Moreover, for the right hand side of \eqref{eq:dydw}, we have 
\beq
\begin{split}
& \sum_{\bsnu \in E_\bsmu} 
{{\bsmu}\choose{\bsnu}}
\int_D \bskappa^{\bsmu-\bsnu} e^{\kappa(\bsy)} \nabla \partial^\bsnu w(\bsy) \cdot \nabla \partial^\bsmu w(\bsy)
dx \\
& \leq \exp\left(||\kappa(\bsy)||_{L^\infty(D)}\right) ||\partial^\bsmu w(\bsy)||_W \sum_{\bsnu \in E_\bsmu} ||\partial^\bsnu w(\bsy)||_W 
{{\bsmu}\choose{\bsnu}} 
||\bskappa||_{L^\infty(D)}^{\bsmu - \bsnu}\\
& \leq \exp\left(||\kappa(\bsy)||_{L^\infty(D)}\right) ||\partial^\bsmu w(\bsy)||_W \sum_{\bsnu \in E_\bsmu} C_\bsnu \exp\left(2|\bsnu| ||\kappa(\bsy)||_{L^\infty(D)}\right) {{\bsmu}\choose{\bsnu}} 
||\bskappa||_{L^\infty(D)}^{\bsmu - \bsnu}\\
& \leq C_\bsmu \exp\left((1+2(|\bsmu|-1))||\kappa(\bsy)||_{L^\infty(D)}\right)  ||\partial^\bsmu w(\bsy)||_W ,
\end{split}
\eeq
where we used Cauchy--Schwarz inequality in the first inequality, the bound \eqref{eq:dydwbound} for $\bsnu \in E_\bsmu$ by induction, and $|\bsnu| \leq |\bsmu|-1$ for $\bsnu \in E_\bsmu$  in the last inequality with the constant 
\beq
C_\bsmu = \sum_{\bsnu \in E_\bsmu} C_\bsnu  {{\bsmu}\choose{\bsnu}} 
||\bskappa||_{L^\infty(D)}^{\bsmu - \bsnu} < \infty,
\eeq
which is independent of $\bsy$.
This bound, together with \eqref{eq:Absmu}, imply \eqref{eq:dydwbound} for $\bsmu$. 
\end{proof}

As a result of Lemma \ref{lem:dydwbound}, the following lemma establishes the relation between the Hermite coefficients and the partial derivatives of the solution.

\begin{lemma}
\label{lem:bnu}
Under Assumption \ref{ass:kappa}, for any $r \in \bbN$, there holds \cite[Theorem 3.1]{BachmayrCohenDeVoreEtAl2017}
\beq
\sum_{||\bsmu||_{\ell^\infty} \leq r} \frac{\bsrho^{2\bsmu}}{\bsmu!} \int_Y ||\partial^\bsmu w(\bsy)||_W^2 d\bsgamma(\bsy) = \sum_{\bsnu \in \cF} b_\bsnu ||w_\bsnu||_W^2,
\eeq
where the coefficients 
\beq\label{eq:bnu}
b_\bsnu := \sum_{||\bsmu||_{\ell^\infty} \leq r} {{\bsnu}\choose{\bsmu}} \bsrho^{2\bsmu} = \prod_{j\geq 1}
\left(
\sum_{l = 0}^r
{{\nu_j}\choose{l}}
\rho_j^{2l}
\right)
, \quad \forall \bsnu \in \cF,
\eeq
which satisfies $(b_\bsnu^{-1})_{\bsnu \in \cF} \in \ell^{q/2}(\cF)$ for $r > \frac{2}{q}$ , see \cite[Lemma 5.1]{BachmayrCohenDeVoreEtAl2017}.
\end{lemma}

By Lemma \ref{lem:bnu}, a weighted summability of the Hermite coefficients of the solution is equivalent to a weighted integrability of the partial derivatives of this solution. The latter is obtained in the following lemma. 

\begin{lemma}
\label{lem:boundeddydw}
Under Assumption \ref{ass:kappa}, by a suitable rescaling of $\bsrho$ such that
\beq
\sup_{x\in D} \sum_{j\geq 1} \rho_j |\kappa_j(x)| =: K <  C_r: = \frac{\ln 2}{\sqrt{r}}
\eeq
for a given $r\in \bbN$, there holds 
\beq
\sum_{||\bsmu||_{\ell^\infty} \leq r} \frac{\bsrho^{2\bsmu}}{\bsmu!} \int_Y ||\partial^\bsmu w(\bsy)||_W^2 d\bsgamma(\bsy) < \infty.
\eeq
\end{lemma}
\begin{proof}
We briefly present the proof following \cite[Theorem 4.1, 4.2]{BachmayrCohenDeVoreEtAl2017}. 

For any integer $k \geq 0$,
let $\Lambda_k := \{\bsmu \in \cF: |\bsmu| = k, ||\bsmu||_{\ell^\infty} \leq r\}$. We define 
\beq
\sigma_k = \sum_{\bsmu \in \Lambda_k} \frac{\bsrho^{2\bsmu}}{\bsmu!} A(\partial^\bsmu w(\bsy), \partial^\bsmu w(\bsy); \kappa(\bsy)).
\eeq
By the equality $B(\partial^\bsmu w(\bsy), \partial^\bsmu w(\bsy)) = 0$ in \eqref{eq:dydw}, we have 
\beq
\sigma_k = -\sum_{\bsmu \in \Lambda_k } \sum_{\bsnu \in E_\bsmu} {{\bsnu}\choose{\bsmu}} \int_D \bskappa^{\bsmu-\bsnu} e^{\kappa(\bsy)} \nabla \partial^\bsnu w(\bsy) \cdot \nabla \partial^\bsmu w(\bsy)
dx.
\eeq
For the right hand side, with suitable algebraic manipulation and using Cauchy--Schwarz inequality, we obtain 
\beq
\sigma_k \leq \sum_{l = 0}^{k - 1} \frac{1}{(k - l)!} (\sqrt{r}K)^{k-l} \sigma_l, 
\eeq
which implies $\sigma_k \leq \sigma_0 \delta^k$ by induction for some $\delta < 1$ and $K/\delta < C_r$. Therefore, 
\beq\label{eq:weightedSumdydw}
\sum_{||\bsmu||_{\ell^\infty} \leq r} \frac{\bsrho^{2\bsmu}}{\bsmu!} A(\partial^\bsmu w(\bsy), \partial^\bsmu w(\bsy); \kappa(\bsy))  \leq \frac{1- \delta^{r+1}}{1-\delta} A(w(\bsy), w(\bsy); \kappa(\bsy)).
\eeq
By the coercivity of $A$ in \eqref{eq:coerciveA} and the bound \eqref{eq:estimateW}, integrating \eqref{eq:weightedSumdydw} we obtain 
\beq
\sum_{||\bsmu||_{\ell^\infty} \leq r} \frac{\bsrho^{2\bsmu}}{\bsmu!} \int_Y ||\partial^\bsmu w(\bsy)||_W^2 d\bsgamma(\bsy) \leq C \bbE[\exp(4 ||\kappa(\bsy)||_{L^\infty(D)})],
\eeq
with constant $C >0 $, which is finite by Proposition \ref{prop:finiteMoments} with $n = 4$.
\end{proof}

By Lemma \ref{lem:bnu} and \ref{lem:boundeddydw}, we have the weighted $\ell^2(\cF)$-summability of the Hermite coefficients of the solution, which leads to their $\ell^s(\cF)$-summability as stated in the following theorem. 

\begin{theorem}
\label{thm:lp}
Under Assumption \ref{ass:kappa}, for the Hermite expansion \eqref{eq:hermite} of the solution of the optimality problem \eqref{eq:pointwiseOC}, there holds 
\beq
\sum_{\bsnu \in \cF} ||w_\bsnu||_W^p < \infty. 
\eeq
\end{theorem}
\begin{proof}
By the H\"older inequality we have
\beq
\sum_{\bsnu \in \cF} ||w_\bsnu||_W^p  \leq \left(\sum_{\bsnu \in \cF} b_\bsnu ||w_\bsnu||_W^2\right)^{p/2} \left(\sum_{\bsnu \in \cF} b_\bsnu^{-q/2}\right)^{1-p/2},
\eeq
which is finite as a result of Lemma \ref{lem:bnu} and \ref{lem:boundeddydw}.
\end{proof}

\section{Sparse quadrature}
\label{sec:sparseQuadrature}
It is often interesting to compute the statistical moments of the control function, or more in general the statistical moments of the  solution of the optimal control problem \eqref{eq:pointwiseOC} or its related quantity of interest. In this section, we present a polynomial-based sparse quadrature \cite{Chen2018} for the computation of these statistical moments, in particular, the expectation of the solution. A dimension-independent convergence rate for the sparse quadrature error will be established based on that of the sparse polynomial approximation. 

\subsection{Sparse quadrature}
We first consider a univariate map $\psi(y) \in X$, where $y \in Y = \bbR$ and has a Gaussian measure $y \sim \gamma =  N(0, 1)$, and $X$ represents a separable Banach space. Our goal is to evaluate the expectation
\beq
\bbE[\psi] = \int_Y \psi(y) d\gamma(y).
\eeq  
To approximation this expectation, we introduce a sequence of univariate quadrature operators $(Q_\nu)_{\nu \geq 0}$ indexed by level $\nu \in \bbN \cup \{0\}$ as 
\beq\label{eq:uQuad}
Q_\nu[\psi] := \sum_{k = 1}^{m(\nu)} \psi(y_k) w_k,
\eeq
where $(y_k, w_k)_{k = 1}^{m(\nu)}$ are the $m(\nu)$ quadrature points and weights. We assume the number of quadrature points $m(\nu)$ satisfies $m(0) = 1$ and $m(\nu) \geq \nu+1$ for $\nu \geq 1$. Classical quadrature points include the non-nested Gauss--Hermite points, and the nested Genz--Keister points, see details in \cite{Chen2018}. 

For any $\nu \geq 0$, the univariate quadrature \eqref{eq:uQuad} can be written in a telescope sum
\beq
Q_\nu[\psi]  = \sum_{l = 0}^\nu \triangle_l[\psi], 
\eeq
where the univariate difference quadrature operator $(\triangle_l)_{l\geq 0}$ are defined as 
\beq
\triangle_l := Q_{l} - Q_{l-1}. 
\eeq
We denote $Q_{-1}[\psi] = 0$ by convention. For a multivariate map $\psi(\bsy) \in X$ where the parameter $\bsy \in Y = \bbR^\bbN$ and has tensorized Gaussian measure $\bsy \sim \bsgamma = \prod_{j\geq 1} \gamma_j$, we define the expectation of $\psi$ as the infinite-dimensional integral
\beq
\bbE[\psi] = \int_Y \psi(\bsy) d\bsgamma(\bsy).
\eeq
To approximate this expectation, based on the univariate quadrature, we define a sparse quadrature $Q_\Lambda$ associated with an index set $\Lambda \subset \cF$ as 
\beq\label{eq:sQuad}
Q_\Lambda[\psi] = \sum_{\bsnu \in \Lambda} \triangle_\bsnu[\psi],
\eeq
where the tensorized difference quadrature operator are defined as
\beq\label{eq:dQuad}
\triangle_\bsnu := \bigotimes_{j\geq 1} \triangle_{\nu_j} = \bigotimes_{j\geq 1} (Q_{\nu_j} - Q_{\nu_j-1}).
\eeq
By increasing the cardinality of the index set $\Lambda$, we hope to obtain a convergent sparse quadrature to the expectation $||\bbE[\psi] - Q_\Lambda[\psi]||_X \to 0$ and to quantify its convergence rate by taking suitable index set $\Lambda \subset \cF$, which are presented in the next section. 

\subsection{Dimension-independent convergence}
To show the convergence and quantify the convergence rate of the sparse quadrature \eqref{eq:sQuad}, we make the following assumptions on the exactness and boundedness of the univariate quadrature $(Q_\nu)_{\nu \geq 0}$. 
\begin{assumption}\label{ass:Quad}
For the univariate quadrature $(Q_\nu)_{\nu \geq 0}$ defined in \eqref{eq:uQuad}, we assume that $Q_\nu$ is exact for polynomials of degree less than or equal to $\nu$, i.e.,
\beq\label{eq:exact}
Q_\nu [\psi] = \bbE[\psi], \quad \forall \psi \in P_\nu(y)=\text{span}\{y^l, l = 0, \dots, \nu\}, \forall \nu \in \bbN\cup\{0\}.
\eeq
In particular, it holds for Hermite polynomials $Q_\nu [H_l] = \bbE[H_l]$, $l = 0, \dots, \nu$. Moreover, we assume that the quadrature values for Hermite polynomials are bounded by
\beq\label{eq:bound}
|Q_\nu[H_l]| \leq 2, \quad \forall l \geq 0, \forall \nu \in \bbN\cup \{0\}.
\eeq
\end{assumption}
\begin{remark}
It is proven in \cite{Chen2018} that the Gauss--Hermite quadrature satisfies Assumption \ref{ass:Quad} for both \eqref{eq:exact} and \eqref{eq:bound}, and the Genz--Keister quadrature satisfies \eqref{eq:exact}. In fact, numerics showed that $|Q_\nu[H_l]| \leq 1$ for both quadrature rules \cite{Chen2018}.
\end{remark}

To this end, we define a specific structure on the index set $\Lambda \subset \cF$, which is called \emph{downward closed} \cite{ChkifaCohenDeVoreEtAl2013}, or \emph{admissible} \cite{GerstnerGriebel2003}, 
\beq
\text{ if } \bsnu \in \Lambda, \text{ then } \bsmu \in \Lambda \text{ for every } \bsmu \preceq \bsnu.
\eeq

\begin{theorem}\label{thm:sQuad}
Under Assumption \ref{ass:kappa} and \ref{ass:Quad}, there exist a sequence of nested and downward closed index sets $(\Lambda_N)_{N\geq 1}$ with cardinality $|\Lambda_N| = N \geq 1$ such that sparse quadrature error for the solution of \eqref{eq:pointwiseOC} is bounded by
\beq
||\bbE[w] - Q_{\Lambda_N} [w]||_W \leq C N^{-s}, \quad s =  \frac{1}{p} - 1, 
\eeq
where the constant $C$ is independent of $N$ and the active parameter dimension. 
\end{theorem}

\begin{proof}
We provide the proof following the arguments in \cite{Chen2018} for the solution of the optimality system \eqref{eq:pointwiseOC}, which are presented in three steps. 

\emph{Step 1}. 
Under the exactness \eqref{eq:exact} in Assumption \ref{ass:Quad} of the univariate quadrature, it is shown \cite[Lemma 3.2]{Chen2018} that for any downward closed index set $\Lambda \subset \cF$, there holds 
\beq\label{eq:exactLambda}
Q_{\Lambda} [\psi] = \bbE[\psi], \quad \forall \psi \in P_{\Lambda}(Y) = \text{span}\{\bsy^\bsnu, \bsnu \in \Lambda\}. 
\eeq 
Moreover, under the boundedness \eqref{eq:bound}, for any $\bsnu \in \cF \setminus \bsnull$, there holds \cite[Lemma 3.2]{Chen2018}
\beq\label{eq:boundedLambda}
|Q_{\Lambda \cap \cR_\bsnu}[H_\bsnu]| \leq \prod_{j\geq 1} (1+\nu_j)^3,
\eeq
where $\cR_\bsnu := \{\bsmu \in \cF: \bsmu \preceq \bsnu\}$. 

\emph{Step 2}. For the solution of \eqref{eq:pointwiseOC}, by the Hermite expansion \eqref{eq:hermite}, we have 
\beq
Q_\Lambda [w] = \sum_{\bsnu \in \Lambda} w_\bsnu Q_\Lambda[H_\bsnu] + \sum_{\bsnu \in \cF \setminus \Lambda} w_\bsnu Q_\Lambda [H_\bsnu].
\eeq
Therefore, by the exactness \eqref{eq:exactLambda}, the sparse quadrature error can be represented by
\beq
\bbE[w] - Q_{\Lambda} [w] = \sum_{\bsnu \in \cF\setminus \Lambda} w_\bsnu (\bbE[H_\bsnu] - Q_{\Lambda}[H_\bsnu]).
\eeq
By the orthonormality of $(H_\bsnu)_{\bsnu \in \cF}$ and $H_\bsnull = 1$, we have $\bbE[H_\bsnu] = \bbE[H_\bsnu H_\bsnull]= 0$ for any $\bsnu \neq \bsnull$, which, together with \eqref{eq:boundedLambda}, lead to the bound for the sparse quadrature error 
\beq\label{eq:cnuw}
||\bbE[w]| - Q_{\Lambda} [w]|_W \leq \sum_{\bsnu \in \cF\setminus \Lambda} c_\bsnu ||w_\bsnu||_W,
\eeq
where $c_\bsnu : = \prod_{j\geq 1} (1+\nu_j)^3$ as given in \eqref{eq:bound}. 

\emph{Step 3}. By referring to the weighted $\ell^2(\cF)$ summability in \eqref{lem:bnu}, we multiply and divide by $b_\bsnu^{-1/2+\eta}$ with $\eta \geq q/4$ for the right hand side of \eqref{eq:cnuw}, which yields
\beq
\sum_{\bsnu \in \cF\setminus \Lambda} c_\bsnu ||w_\bsnu||_W \leq \sup_{\bsnu \in \cF\setminus \Lambda} b_\bsnu^{-1/2+\eta} \sum_{\bsnu \in \cF\setminus \Lambda} \frac{c_\bsnu}{b_\bsnu^\eta} b_\bsnu^{1/2} ||w_\bsnu||_W.
\eeq
Using Cauchy--Schwarz inequality, we have 
\beq
\sum_{\bsnu \in \cF\setminus \Lambda} \frac{c_\bsnu}{b_\bsnu^\eta} b_\bsnu^{1/2} ||w_\bsnu||_W \leq \left(\sum_{\bsnu \in \cF\setminus \Lambda} \left(\frac{c_\bsnu}{b_\bsnu^\eta}\right)^2 \right)^{1/2}\left(\sum_{\bsnu \in \cF\setminus \Lambda} b_\bsnu ||w_\bsnu||_W^2 \right)^{1/2},
\eeq
where the second term is finite by Lemma \ref{lem:bnu} and \ref{lem:boundeddydw}, the first term is also finite as shown in \cite[Lemma 3.5]{Chen2018}. Moreover, as shown in  \cite[Theorem 3.6]{Chen2018}, for a sequence of nested multi-index sets $(\Lambda_N)_{N\geq 1}$ with elements corresponding to the indices of the $N$ largest $b_\bsnu^{-1/2+\eta}$ among all $\bsnu \in \cF$ , there holds 
\beq
\sup_{\bsnu \in \cF\setminus \Lambda_N} b_\bsnu^{-1/2+\eta} \leq \left(\sum_{\bsnu \in \cF} b_\bsnu^{-q/2}\right)^{(2-q)/2q} (N+1)^{-s}, \quad s = \frac{1}{q} - \frac{1}{2} = \frac{1}{p} - 1,
\eeq
with $\eta = q/4$. Furthermore, it is shown that the sequence $(b_\bsnu)_{\bsnu \in \cF}$ is monotonically increasing, i.e., $b_\bsnu \leq b_\bsmu$ for $\bsnu \preceq \bsmu$, so that $(b_\bsnu^{-1/2+\eta})_{\bsnu \in \cF}$ is monotonically decreasing with $\eta < 1/2$, which implies that $(\Lambda_N)_{N\geq 1}$ can be taken downward closed.
\end{proof}

\begin{remark}
Note that the dimension-independent convergence is with respect to the number of indices $N$ in $\Lambda_N$. As the computational complexity depends on the number of PDE solves, or the number of quadrature points $N_p$ in $\Lambda_N$, which scales as $N_p \leq C N^2$ for Gauss--Hermite quadrature \cite[Proposition 18]{ErnstSprungkTamellini2018}, the corresponding sparse quadrature error is therefore bounded as 
\beq
||\bbE[w] - Q_{\Lambda_N} [w]||_W \leq C N_p^{-s/2}, \quad s = \frac{1}{p} - 1.
\eeq
This bound, however, is likely not optimal since the Gauss--Hermite quadrature $Q_\nu[\psi]$ is exact for $\psi \in P_{2m(\nu)-1} = \text{span}\{y^{l}, l = 0, \dots, 2m(\nu)-1\}$ where $m(\nu) \geq \nu + 1$. 
\end{remark}
\subsection{A-priori and a-posteriori construction algorithms}
\label{subsec:errorIndicator}
Theorem \ref{thm:sQuad} states the existence of a sequence of nested and downward closed index sets $(\Lambda_N)_{N\geq 1}$ in achieving the dimension-independent convergence rate of the sparse quadrature error. To construct such index sets, we present an adaptive algorithm that was originally developed in \cite{GerstnerGriebel2003} and use both a-priori and a-posteriori error indicators as further developed in \cite{Chen2018}. 

As the index sets $(\Lambda_N)_{N\geq 1}$ are downward closed and nested, we can start from the root $\Lambda = \{\bsnull\}$ and in each step enrich the index set by one of the indices from its reduced forward neighbor set defined as \cite{SchillingsSchwab2013, Chen2018} 
\beq\label{eq:ForwNeib}
\cN(\Lambda) = \{\bsnu \in \cF\setminus \Lambda: \bsnu-\bse_j \in \Lambda, \;\forall j \in \bbJ_\bsnu \text{ and } \nu_j = 0, \; \forall j > j(\Lambda)+1\},
\eeq
where $\bbJ_\bsnu := \{j\geq 1: \nu_j \neq 0\}$, and $j(\Lambda)$ is the smallest $j$ such that $\nu_{j+1} = 0$ for all $\bsnu \in \Lambda$. Then in each step, we select the next index $\bsnu \in \cN(\Lambda)$ according to an a-priori error indicator $b_\bsnu$ as illustrated in the proof, Step 3, of Theorem \ref{thm:sQuad}, or an a-posteriori error indicator $||\triangle_\bsnu[\psi]||_X$ as defined in \eqref{eq:dQuad}. The adaptive construction process for the sparse quadrature is presented in the following algorithm, where a maximum number of indices is prescribed as a stopping criterion. Alternatively, a prescribed tolerance for the error indicators can be likewise imposed.

\begin{algorithm}
\caption{Adaptive sparse quadrature}
\label{alg:SparseQuad}
\begin{algorithmic}[1]
\STATE{\textbf{Input: } 
a maximum number of indices $N_{\text{max}}$, the map $\psi$.}
\STATE{\textbf{Output: } the downward closed index set $\Lambda_N$, quadrature $\cQ_{\Lambda_N}(\psi)$.}
\STATE{Set $N= 1$, $\Lambda_N = \{\bsnull\}$, evaluate $\psi(\bsnull)$ and set $
\cQ_{\Lambda_N}(\psi)=\psi(\bsnull)$.}
\WHILE{$N < N_{\text{max}}$}
\STATE{Construct the reduced forward neighbor set $\cN(\Lambda_N)$ by \eqref{eq:ForwNeib}.}
\IF {using a-priori construction}
\STATE{Compute $b_{\bsnu}$ for all $\bsnu \in \cN(\Lambda_N)$ by \eqref{eq:bnu}.}
\STATE{Take $\bsnu = \argmin_{\bsmu \in \cN(\Lambda_N)} b_\bsnu$.}
\ELSIF {using a-posteriori construction}
\STATE{Compute $\triangle_\bsnu(\psi)$ for all $\bsnu \in \cN(\Lambda_N)$ by \eqref{eq:dQuad}.}
\STATE{Take $\bsnu = \argmax_{\bsmu \in \cN(\Lambda_N)} ||\triangle_{\bsmu}(\psi)||_X$.}
\ENDIF
\STATE{Enrich the index set $\Lambda_{N+1} = \Lambda_N \cup \{\bsnu\}$}.
\STATE{Set $\cQ_{\Lambda_{N+1}}(\psi) = \cQ_{\Lambda_N}(\psi) + \triangle_{\bsnu}(\psi)$.}
\STATE{Set $ N \leftarrow N + 1$.}
\ENDWHILE
\end{algorithmic}
\end{algorithm}

\begin{remark}
By arranging the sequence $(\rho_j)_{j\geq 1}$ that satisfies Assumption \ref{ass:kappa} in an increasing order with respect to $j$, the adaptive construction in the reduced forward neighbor set with the a-priori error indicator is guaranteed to achieve the convergence rate in Theorem \ref{thm:sQuad}, see more details in \cite{Chen2018}. On the other hand, the a-posteriori error indicator does not guarantee to achieve such convergence rate in theory but lead to smaller quadrature error than the a-priori error indicator in practice. 
\end{remark}

\begin{remark}\label{rmk:prioriposteriori}
In the construction by the a-priori error indicator, one only needs to compute $b_\bsnu$ in seeking the next index in the forward neighbor set $\bsnu \in \cN(\Lambda_N)$, which may greatly reduce the computation cost compared to that by the a-posteriori error indicator, which requires PDE solves to compute $\triangle_\bsnu[\psi]$ for all $\bsnu \in \cN(\Lambda_N)$.
\end{remark}

\section{Numerical experiment}
\label{sec:numerics}
In this section, we perform a numerical experiment to demonstrate the dimension-independent convergence property of the sparse quadrature for the optimal control problem with lognormal random coefficients. We consider the optimal control problem \eqref{eq:stochasticOC} in one dimensional physical domain $D = (0, 1)$ and impose homogeneous Dirichlet boundary condition for the elliptic PDE \eqref{eq:ellipticControl}. We use a finite element method with piecewise linear element in a uniform mesh with $1025$ nodes for discretization in the physical space. 
We specify the parametrization \eqref{eq:kappa} as
\beq\label{eq:kappa1025}
\kappa(x,\bsy) = \sum_{j = 1}^{\infty} y_j j^{-\alpha} \sin(\pi j x)/2, 
\eeq
where $\alpha \geq 1$. Due to the physical discretization, \eqref{eq:kappa1025} is truncated with 1025 terms, resulting in a 1025-dimensional stochastic optimal control problem.
For this parametrization, we can take $\bsrho = j^{\alpha-1-\epsilon}$ for arbitrarily small $\epsilon > 0$, so that \eqref{eq:rhokappa} is satisfied and $(\rho_j^{-1})_{j\geq 1} \in \ell^q$ for any $q > \frac{1}{\alpha - 1}$. Therefore, by Theorem \ref{thm:bestN}, the dimension-independent convergence rate $N^{-s}$ can be obtained with $s < \alpha - 1$ for the sparse polynomial approximation, and with $s < \alpha - 3/2$ for the sparse quadrature by Theorem \ref{thm:sQuad}. We generate the synthetic data $u_d$ as the solution of \eqref{eq:ellipticControl} with $z_d = \sin(\pi x)$ and $\kappa = 0$. We take the regularization parameter $\beta = 10^{-4}$. 

\begin{figure}[!htb]
\begin{center}
\includegraphics[scale=0.32]{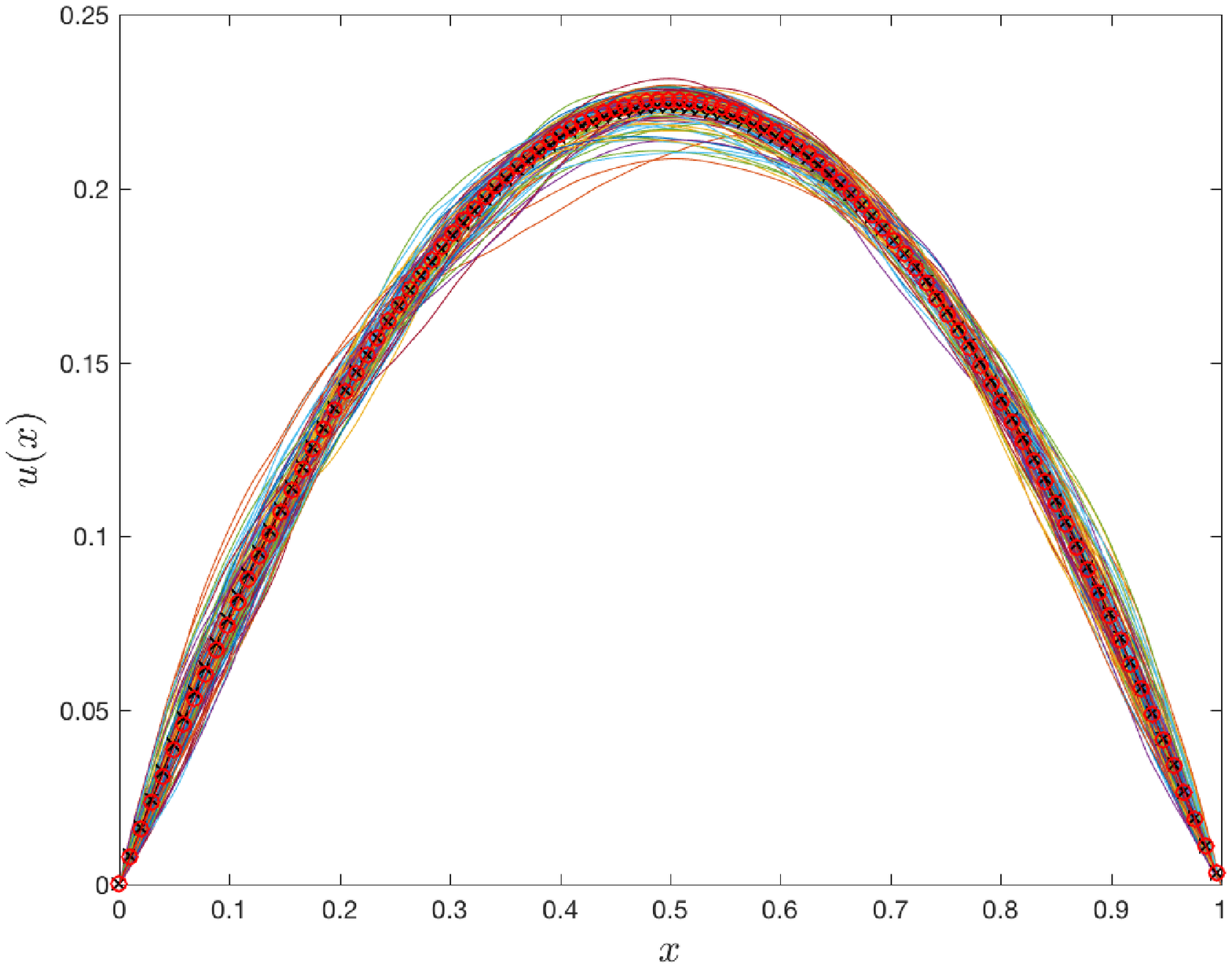}
\includegraphics[scale=0.32]{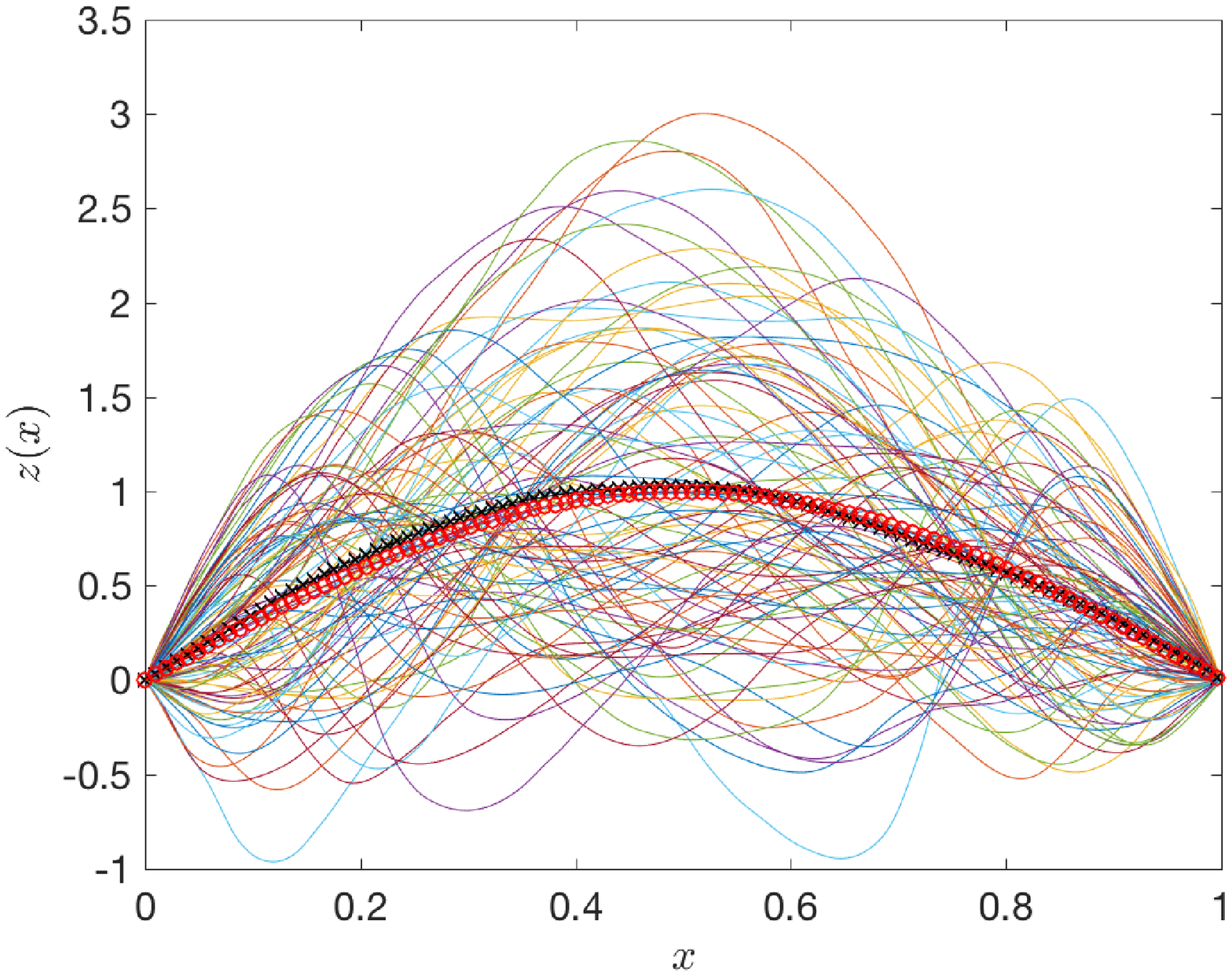}

\includegraphics[scale=0.32]{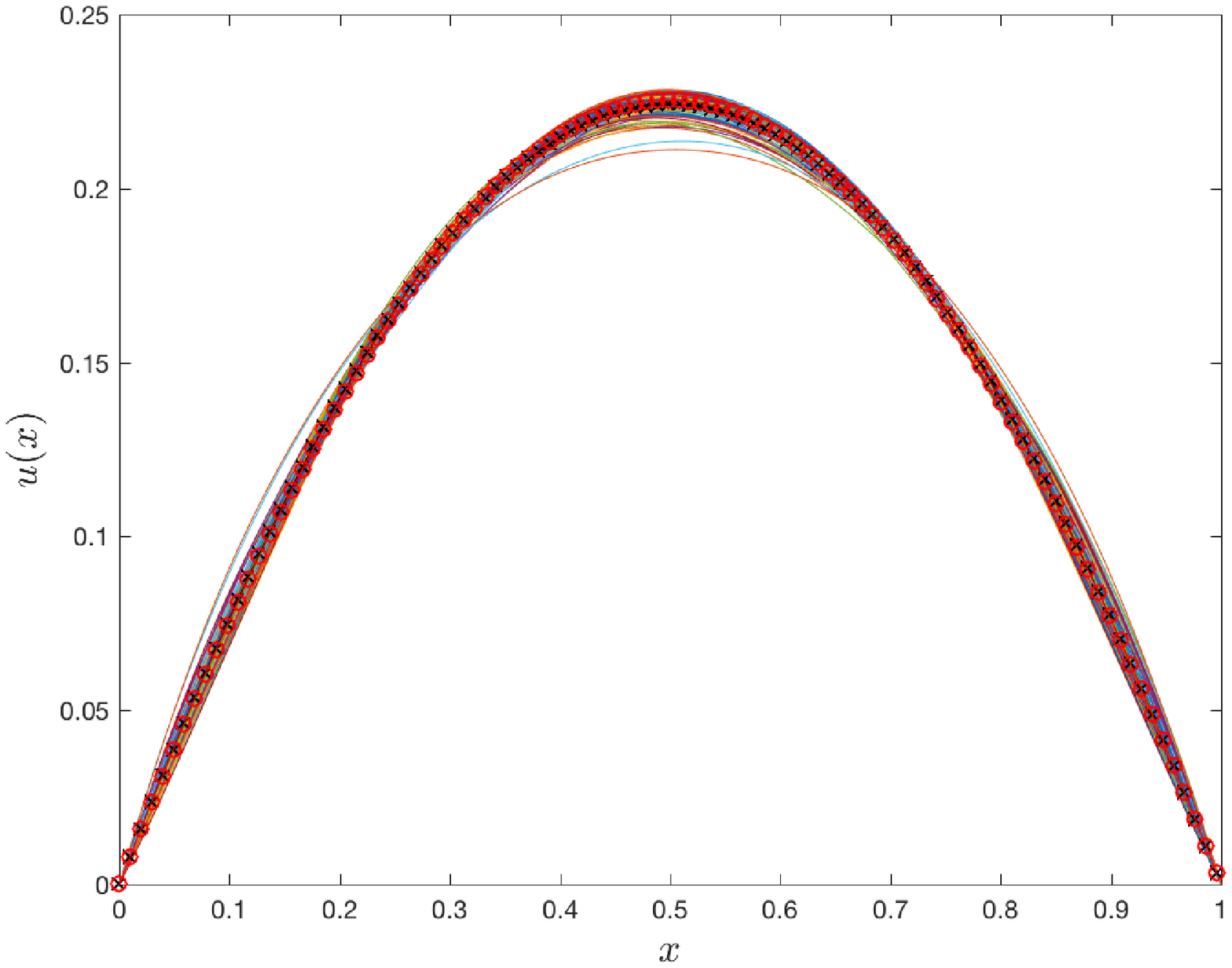}
\includegraphics[scale=0.32]{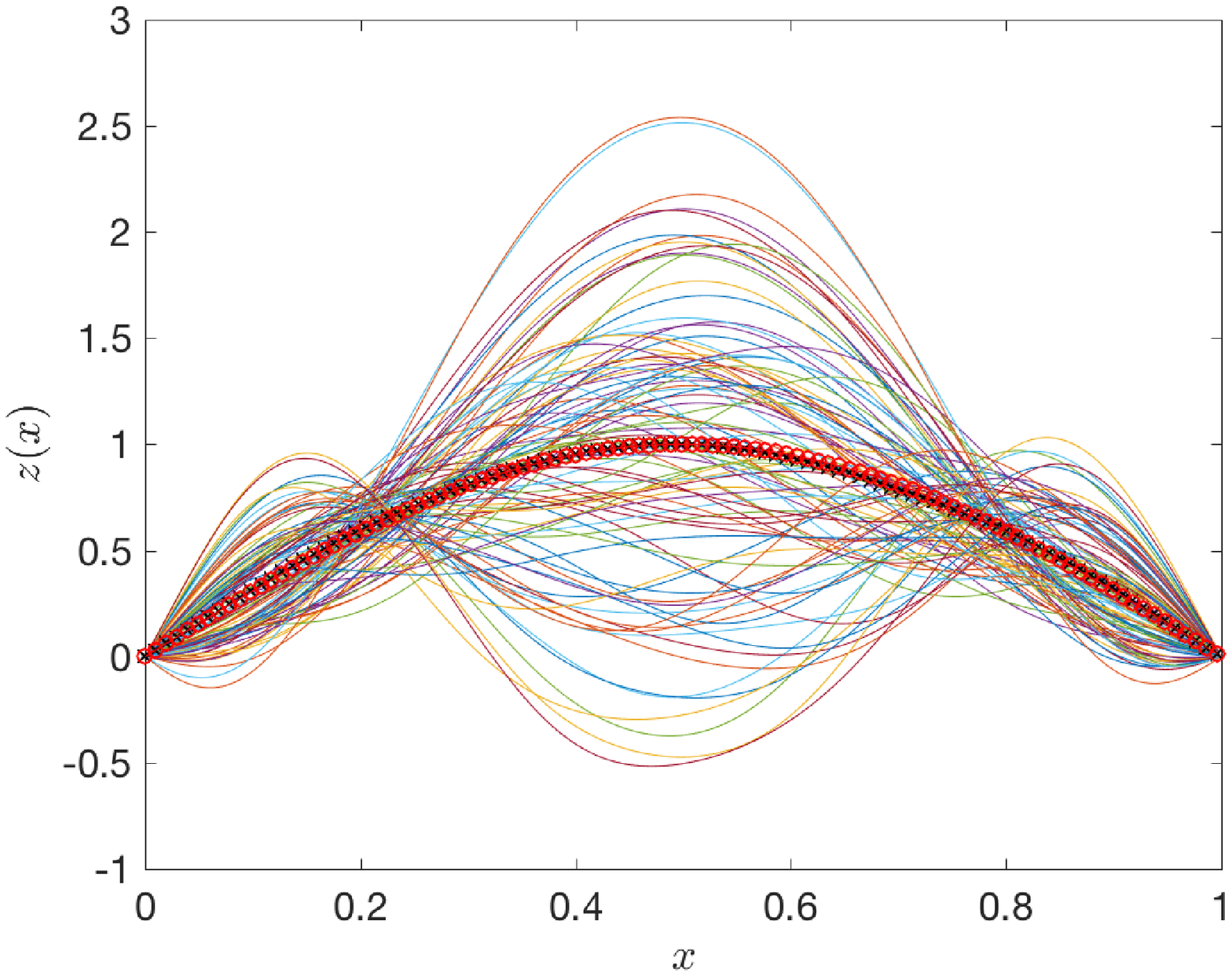}
\end{center}
\caption{The state variable $u$ (left) and the control variable $z = -v/\beta$ (right) of the optimality system \eqref{eq:pointwiseOC} at 100 random samples. The lines marked by red circle represent the synthetic data $u_d$ at the control $z_d = \sin(\pi x)$, the line marked by black cross represent the averaged state and control among the 100 realizations. $\alpha = 1$ (top), $\alpha = 2$ (bottom).
}\label{fig:controlSample}
\end{figure}

At first, we take 100 random samples of $\kappa$ with $\alpha = 1, 2$ and solve the optimality system \eqref{eq:pointwiseOC} at these samples. The realizations of the state variable $u$ and the control variable given through the adjoint variable as $z = -v/\beta$ are plotted in Fig.\ \ref{fig:controlSample}. We can observe that the realizations of the state variable are close to the synthetic data $u_d$ as expected (the difference between them is the object to minimize in the cost functional), while the realizations of the control variable are quite far from $z_d$ at different random samples. Nevertheless, the sample mean of the control variable is rather close to $z_d$ for both cases of $\alpha = 1, 2$. This observation implies that even using a stochastic control function, its mean may be used to approximate a deterministic control as argued in \cite{AliUllmannHinze2017}.

Next, we use the sparse quadrature based on univariate Gauss--Hermite quadrature rule to compute the expectation of the control variable $z(x)$ at $x = 0.5$. For the construction of the sparse quadrature, we run Algorithm \ref{alg:SparseQuad} using both a-priori and a-posteriori error indicators presented in Sec. \ref{subsec:errorIndicator}. A maximum number of quadrature points $N_{\text{max}} = 10^4$ is prescribed as the stopping criterion. We consider three different cases $\alpha = 1, 2, 3$ in the parametrization \eqref{eq:kappa1025}. 
The convergence of the sparse quadrature errors in different scenarios are shown in Fig.\ \ref{fig:convergence} with the reference values computed at the maximum number of quadrature points. 

\begin{figure}[!htb]
\begin{center}
\includegraphics[scale=0.33]{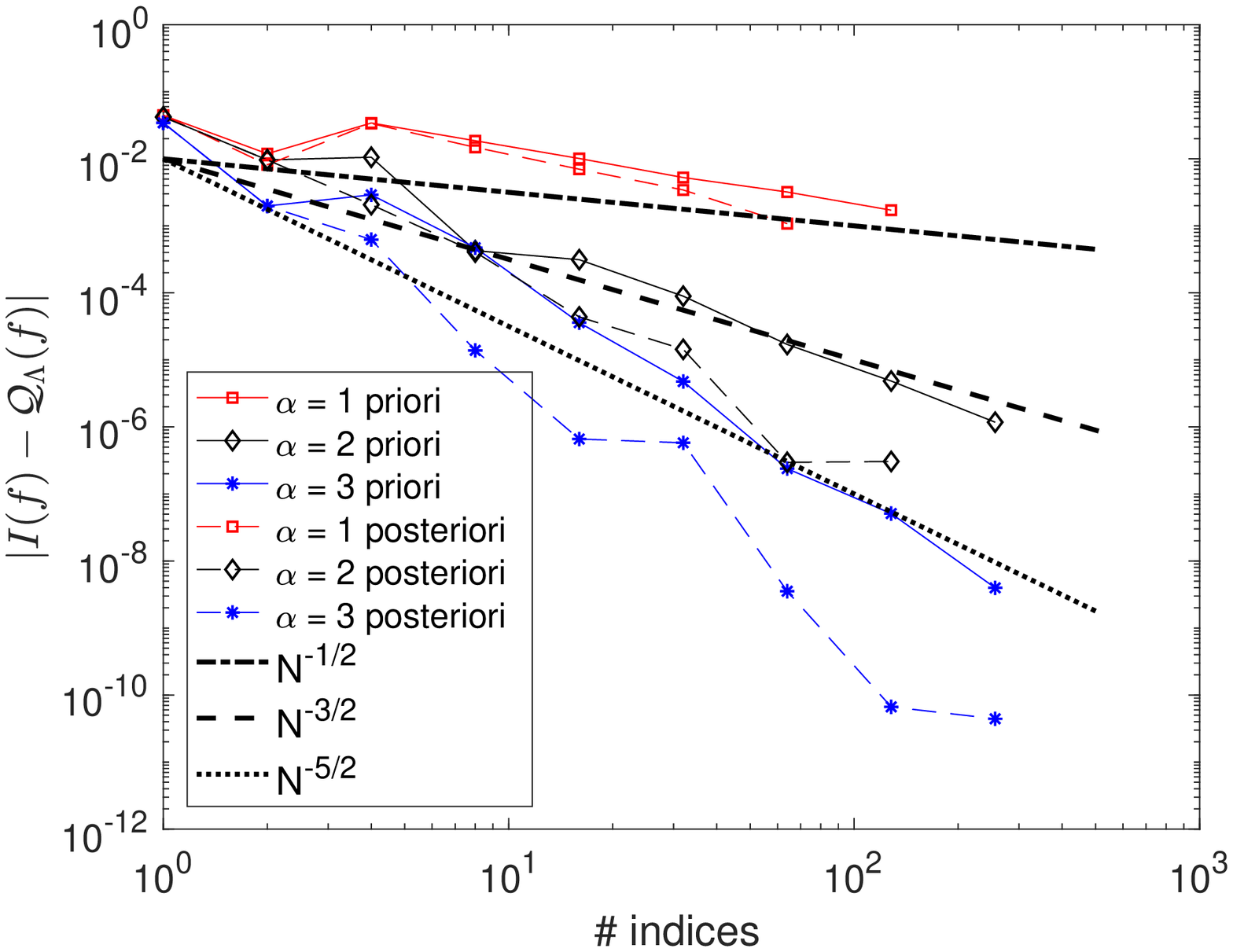}
\includegraphics[scale=0.33]{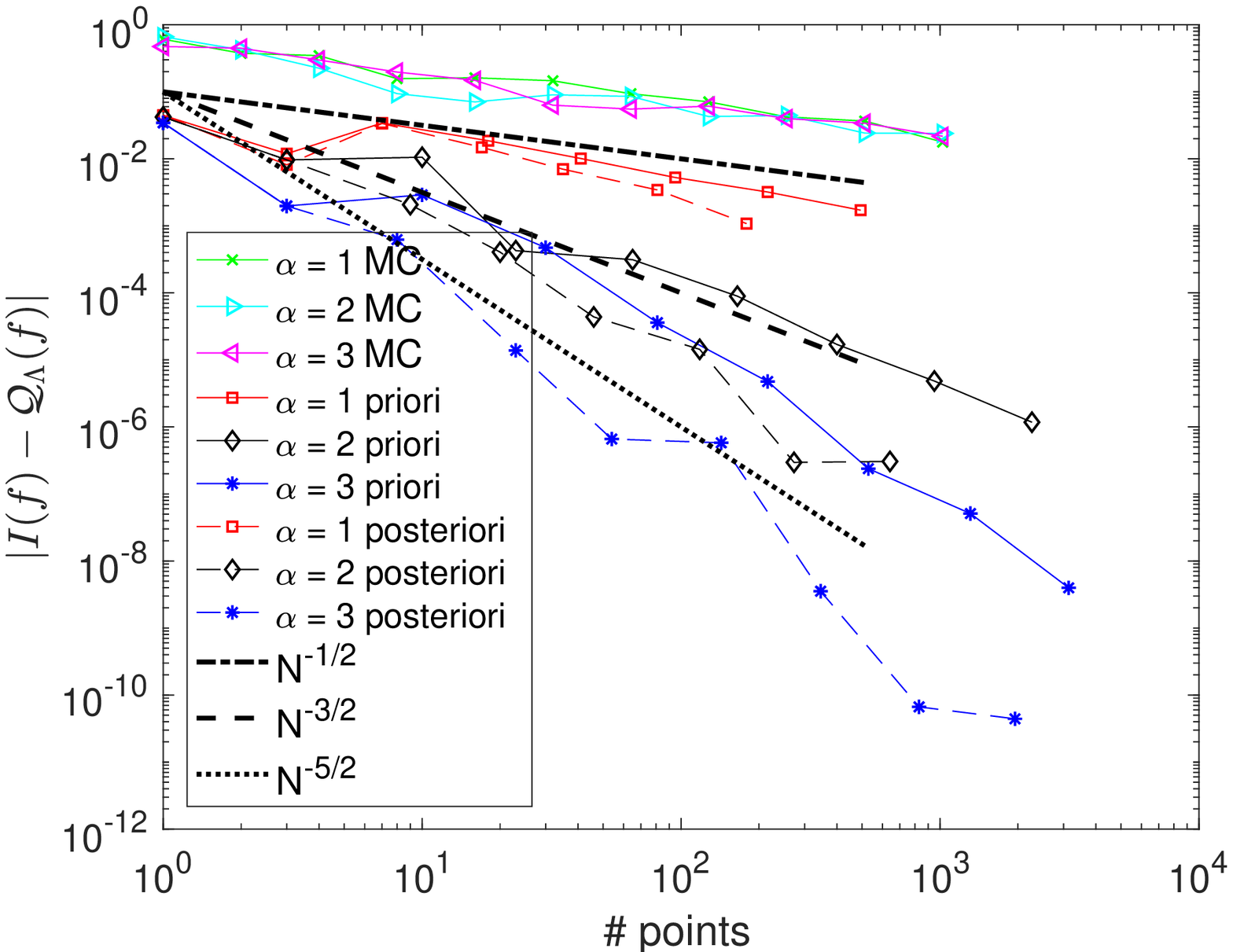}

\includegraphics[scale=0.33]{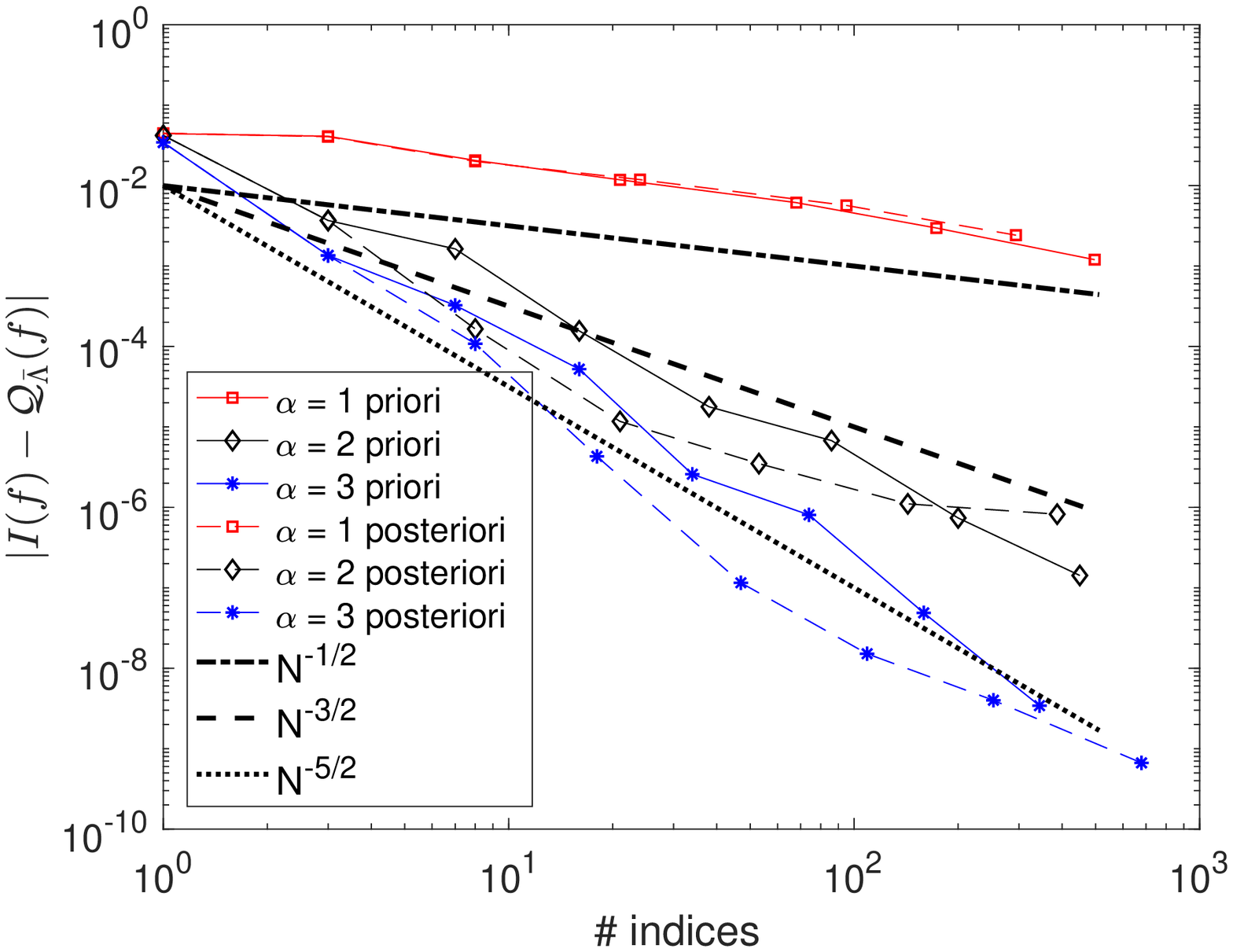}
\includegraphics[scale=0.33]{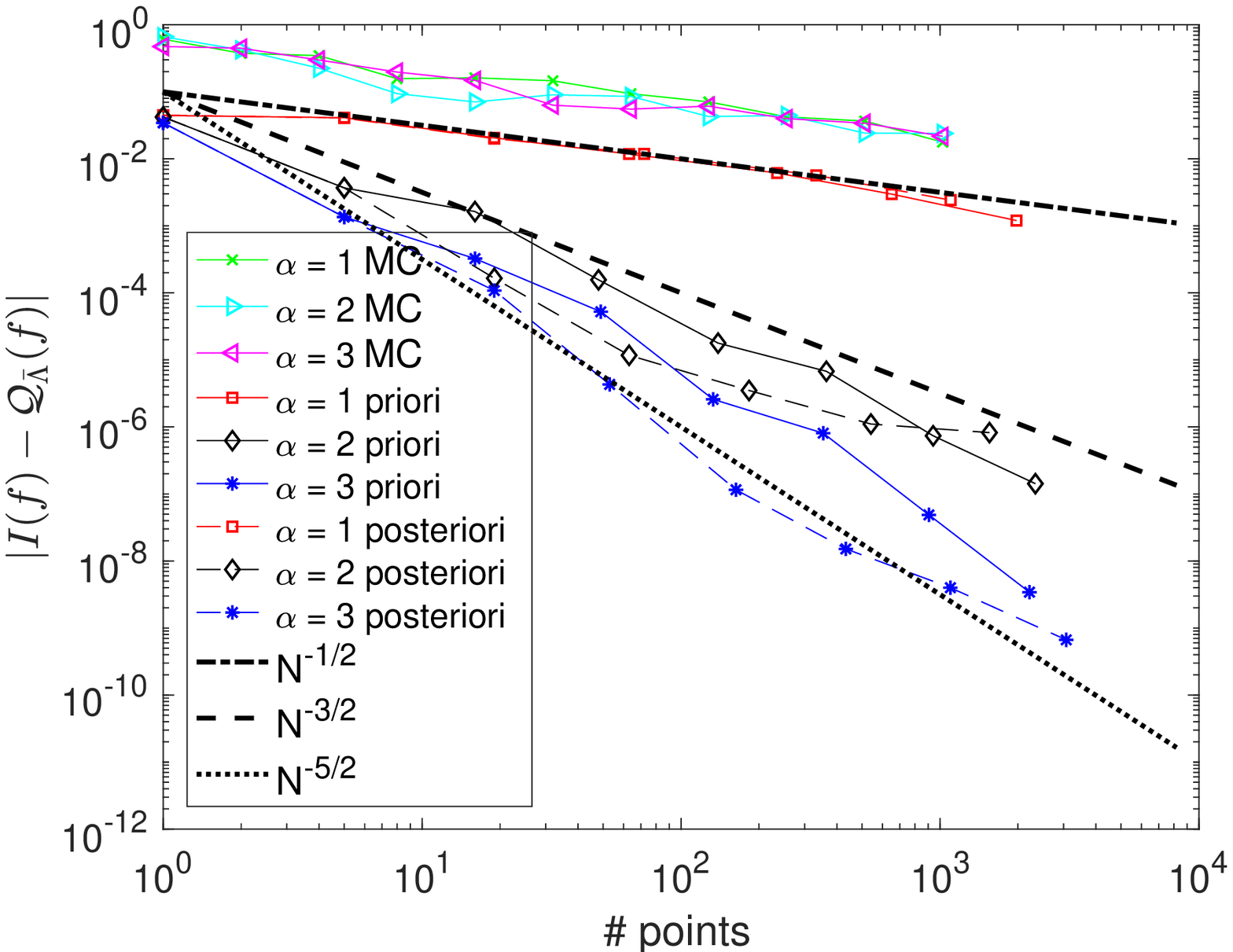}
\end{center}
\caption{Convergence of the sparse quadrature errors with respect to the number of indices (left) and the number of quadrature points or PDE solves (right). In the top figures, the convergence is shown in $\Lambda$, while in the bottom figures, the convergence is shown in $\bar{\Lambda} = \Lambda \cup \cN(\Lambda)$. In the right figures, convergence of Monte Carlo quadrature errors is also shown. $\alpha = 1, 2, 3$.}\label{fig:convergence}
\end{figure}

First, from the top-left part of Fig.\ \ref{fig:convergence}, we can see that the sparse quadrature errors converge with an asymptotic rate of $N^{-s}$ with respect to the number of indices $N$ in the constructed index set $\Lambda$. Varying $\alpha = 1, 2, 3$, we can observe the rate $s > \alpha - 1/2$ for both the a-priori and the a-posteriori constructions, which is larger than $\alpha - 3/2$ as predicted by Theorem \ref{thm:sQuad}. This larger convergence rate has also been observed in \cite{Chen2018} for Gaussian measure and in \cite{SchillingsSchwab2013} for uniform measure. We mention that the theoretical prediction has been improved in a recent work \cite{ZechSchwab2017} for uniform measure, from which we may expect an improvement for the Gaussian measure. Note that the a-posteriori error indicator results in smaller quadrature errors than those produced by the a-priori error indicator, even though the former can not guarantee the theoretical prediction. The convergence of the sparse quadrature errors with respect to the number of quadrature points in $\Lambda$ are shown in the top-right part of Fig.\ \ref{fig:convergence}, which is slightly slower than that with respect to the number of indices, yet still implies a rate of $s \geq \alpha - 1/2$. The convergence for the averaged Monte Carlo quadrature errors with 10 trials is also shown in this figure. It is evident that the averaged asymptotic convergence rate of the Monte Carlo quadrature errors is about $N^{-s}$ with $s = 1/2$ for all cases of $\alpha = 1, 2, 3$, which is smaller that that of the sparse quadrature errors for $\alpha = 2, 3$. Even in the case of $\alpha = 1$, the sparse quadrature errors are smaller than the Monte Carlo quadrature errors. In the construction by the a-posteriori error indicator, the total computational cost include that for PDE solves at the quadrature points corresponding to the indices in the forward neighbor set $\cN(\Lambda)$, as mentioned in Remark \ref{rmk:prioriposteriori}. Therefore, it is important to consider the convergence with respect to this total computational cost in terms of PDE solves (quadrature points) in $\bar{\Lambda} = \Lambda \cup \cN(\Lambda)$, which is displayed in the bottom part of Fig.\ \ref{fig:convergence}. Again, we can observe the convergence rate $s \geq \alpha - 1/2$ for the sparse quadrature errors with respect to the number of both indices and quadrature points, which is faster than that of Monte Carlo for $\alpha = 2, 3$.  

Finally, in Fig.\ \ref{fig:level} we plot the maximum levels of indices $l_j = \max_{\bsnu \in \Lambda \cup \cN(\Lambda)} \nu_j$ in each parameer dimension $j = 1, \dots, 1025$, among all indices $\bsnu \in \Lambda \cup \cN(\Lambda)$. Note that the dimension $j$ is activated in $\Lambda$ once $l_j \geq 2$, and in $\Lambda \cup \cN(\Lambda)$ once $l_j \geq 1$. From the plot we can observe that several hundred of dimensions are activated in $\Lambda \cup \cN(\Lambda)$ while less than one hundred dimensions are activated in $\Lambda$. As the reference quadrature value used in Fig.\ \ref{fig:convergence} is computed using all the indices in $\Lambda \cup \cN(\Lambda)$, it implies that the asymptotic convergence rates displayed in Fig.\ \ref{fig:convergence} are indeed dimension-independent. Moreover, from Fig.\ \ref{fig:level} we can also observe that more dimensions with bigger maximum levels in the first few dimensions are activated for $\alpha = 2$ than for $\alpha = 1$, which agrees with the fact that the former has stronger anisotropicity. Furthermore, we remark that once the parametrization \eqref{eq:kappa1025} is fixed, the a-priori error indicator will produce the same index set regardless of the integrand, while for the a-posteriori error indicator, the constructed index set depends on the integrand and likely results in smaller quadrature errors.

\begin{figure}[!htb]
\begin{center}
\includegraphics[scale=0.33]{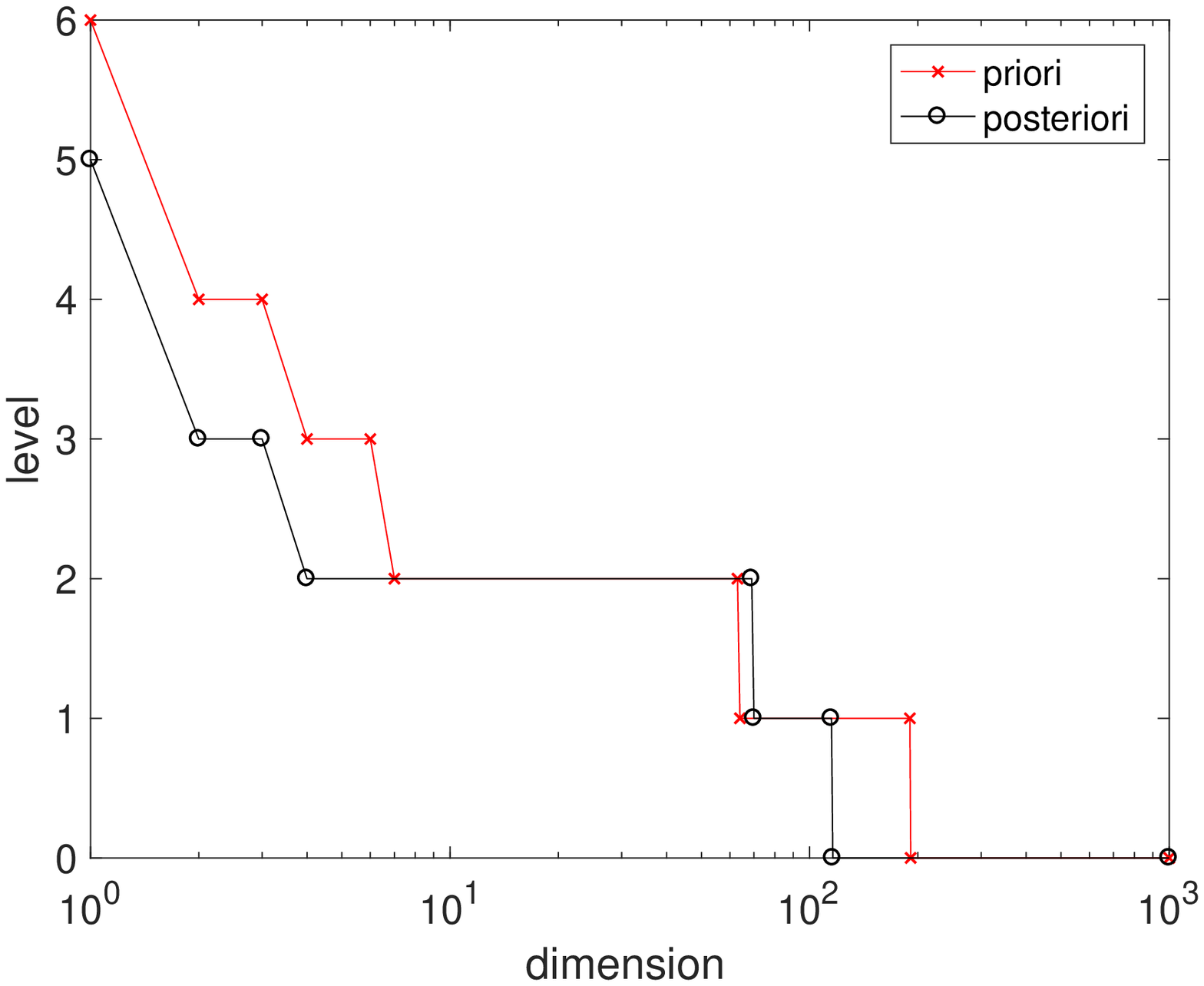}
\hspace*{.1cm}
\includegraphics[scale=0.33]{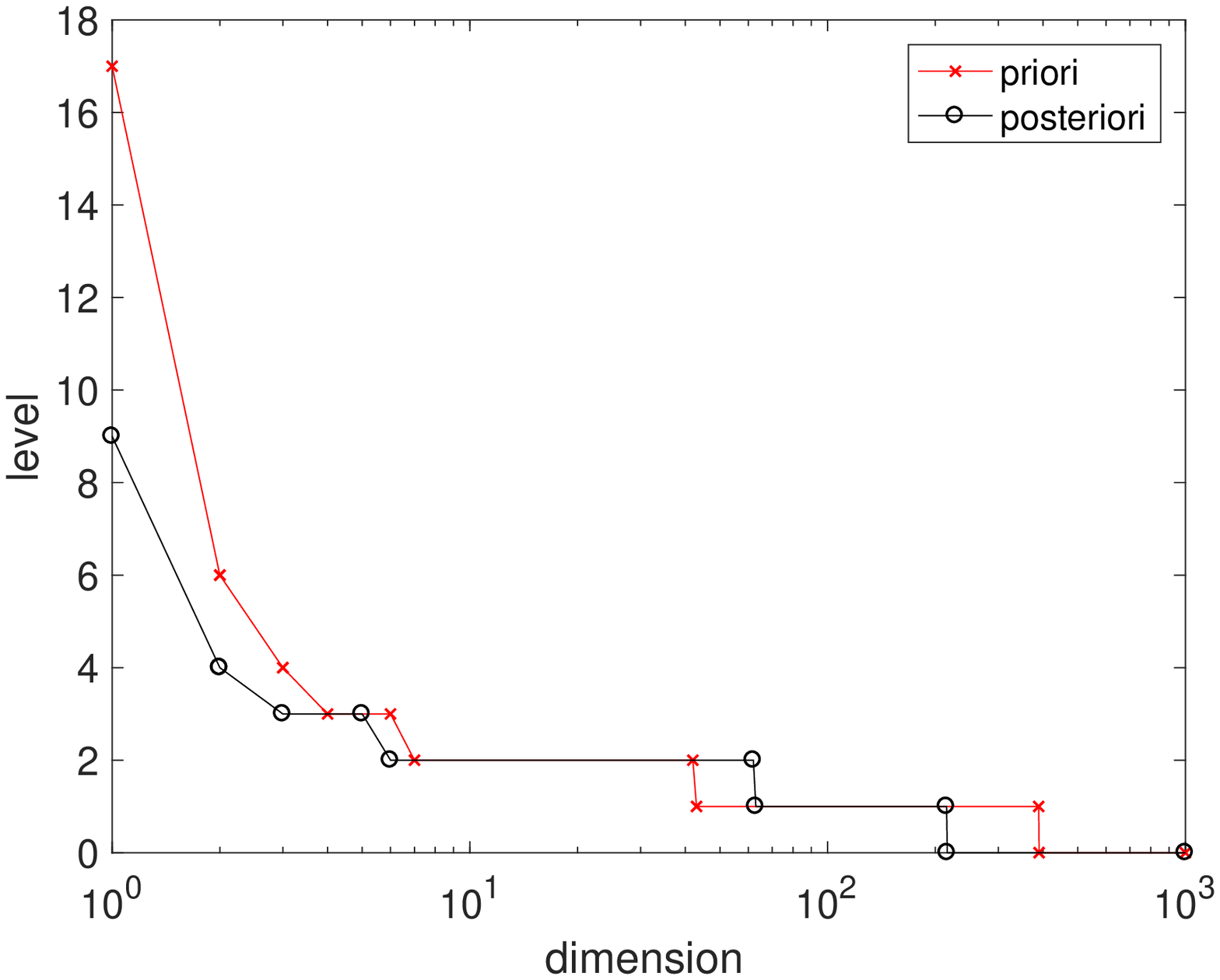}
\end{center}
\caption{Comparison between a-priori and a-posteriori constructions for maximum level of indices constructed in each dimension for indices in $\bar{\Lambda} = \Lambda \cup \cN(\Lambda)$. $\alpha = 1$ (left), $\alpha = 2$ (right).}\label{fig:level}
\end{figure}

\section{Conclusion}
\label{sec:conclusion}
In this work, we proposed and analysed a sparse polynomial approximation for the solution of optimal control problems constrained by elliptic PDEs with lognormal random coefficients. Under certain assumptions on the infinite-dimensional parametrization of the lognormal random field, we proved that the convergence rate of the sparse polynomial approximation of the optimal solution is dimension-independent. For the computation of the expectation of the optimal solution, we presented a polynomial-based sparse quadrature as a sum of tensorized univariate quadrature in a downward closed index set. Given the convergence property of the sparse polynomial approximation, we also established the convergence rate of the sparse quadrature under assumptions of exactness and boundedness of the univariate quadrature rule. Numerical experiments for a 1025-dimensional optimal control problem confirmed that the convergence rate of the sparse quadrature is independent of the number of active parameter dimensions. Moreover, the convergence can be much faster than that of Monte Carlo quadrature provided the lognormal random field is sufficiently smooth or the modes of its parametrization decay sufficiently quickly. The optimal control problem we considered is relatively simple, in the sense that the control function is distributed and parameter-dependent, and the cost functional involves only the expectation as the risk measure. Further analysis and application of the sparse polynomial approximation and sparse quadrature are desirable for more general optimal control problems with other types of control functions, risk measures, and PDE constraints. 

\bibliographystyle{plain}
\bibliography{bibliographyAuth3Year.bib}

\begin{thebibliography}{10}

\bibitem{AlexanderianPetraStadlerEtAl2017}
Alen Alexanderian, Noemi Petra, Georg Stadler, and Omar Ghattas.
\newblock Mean-variance risk-averse optimal control of systems governed by
  {PDEs} with random parameter fields using quadratic approximations.
\newblock {\em SIAM/ASA Journal on Uncertainty Quantification},
  5(1):1166--1192, 2017.

\bibitem{AliUllmannHinze2017}
Ahmad~Ahmad Ali, Elisabeth Ullmann, and Michael Hinze.
\newblock Multilevel {Monte Carlo} analysis for optimal control of elliptic
  {PDEs} with random coefficients.
\newblock {\em SIAM/ASA Journal on Uncertainty Quantification}, 5(1):466--492,
  2017.

\bibitem{BachmayrCohenDeVoreEtAl2017}
Markus Bachmayr, Albert Cohen, Ronald DeVore, and Giovanni Migliorati.
\newblock {Sparse polynomial approximation of parametric elliptic PDEs. Part
  {II}: {L}ognormal coefficients}.
\newblock {\em ESAIM: Mathematical Modelling and Numerical Analysis},
  51(1):341--363, 2017.

\bibitem{BachmayrCohenMigliorati2016}
Markus Bachmayr, Albert Cohen, and Giovanni Migliorati.
\newblock Representations of {G}aussian random fields and approximation of
  elliptic {PDEs} with lognormal coefficients.
\newblock {\em Journal of Fourier Analysis and Applications}, pages 1--29,
  2016.

\bibitem{BennerDolgovOnwuntaEtAl2016}
Peter Benner, Sergey Dolgov, Akwum Onwunta, and Martin Stoll.
\newblock Low-rank solvers for unsteady {Stokes--Brinkman} optimal control
  problem with random data.
\newblock {\em Computer Methods in Applied Mechanics and Engineering},
  304:26--54, 2016.

\bibitem{BennerOnwuntaStoll2016}
Peter Benner, Akwum Onwunta, and Martin Stoll.
\newblock Block-diagonal preconditioning for optimal control problems
  constrained by {PDEs} with uncertain inputs.
\newblock {\em SIAM Journal on Matrix Analysis and Applications},
  37(2):491--518, 2016.

\bibitem{Borzi2010}
Alfio Borz{\`\i}.
\newblock Multigrid and sparse-grid schemes for elliptic control problems with
  random coefficients.
\newblock {\em Computing and visualization in science}, 13(4):153--160, 2010.

\bibitem{BorziSchulz2011}
Alfio Borz\`i and Volker Schulz.
\newblock {\em Computational optimization of systems governed by partial
  differential equations}, volume~8.
\newblock SIAM, 2011.

\bibitem{BorziSchulzSchillingsEtAl2010}
Alfio. Borz{\`i}, Volker Schulz, Claudia Schillings, and Gregory von Winckel.
\newblock On the treatment of distributed uncertainties in {PDE}-constrained
  optimization.
\newblock {\em GAMM-Mitteilungen}, 33(2):230--246, 2010.

\bibitem{BorziWinckel2011}
Alfio Borz{\`\i} and Gregory von Winckel.
\newblock A {POD} framework to determine robust controls in {PDE} optimization.
\newblock {\em Computing and visualization in science}, 14(3):91--103, 2011.

\bibitem{Charrier2012}
Julia Charrier.
\newblock Strong and weak error estimates for elliptic partial differential
  equations with random coefficients.
\newblock {\em SIAM Journal on numerical analysis}, 50(1):216--246, 2012.

\bibitem{Chen2018}
Peng Chen.
\newblock Sparse quadrature for high-dimensional integration with {G}aussian
  measure.
\newblock {\em ESAIM: Mathematical Modelling and Numerical Analysis},
  52(2):631--657, 2018.

\bibitem{ChenQuarteroni2014}
Peng Chen and Alfio Quarteroni.
\newblock Weighted reduced basis method for stochastic optimal control problems
  with elliptic {PDE} constraints.
\newblock {\em SIAM/ASA J. Uncertainty Quantification}, 2(1):364--396, 2014.

\bibitem{ChenQuarteroniRozza2013}
Peng Chen, Alfio Quarteroni, and Gianluigi Rozza.
\newblock Stochastic optimal {R}obin boundary control problems of
  advection-dominated elliptic equations.
\newblock {\em SIAM Journal on Numerical Analysis}, 51(5):2700 -- 2722, 2013.

\bibitem{ChenQuarteroniRozza2016}
Peng Chen, Alfio Quarteroni, and Gianluigi Rozza.
\newblock Multilevel and weighted reduced basis method for stochastic optimal
  control problems constrained by {S}tokes equations.
\newblock {\em Numerische Mathematik}, 133(1):67--102, 2016.

\bibitem{ChenVillaGhattas17}
Peng Chen, Umberto Villa, and Omar Ghattas.
\newblock Taylor approximation and variance reduction for {PDE}-constrained
  optimal control problems under uncertainty.
\newblock {\em Journal of Computational Physics}, 385:163--186, 2019.

\bibitem{ChkifaCohenDeVoreEtAl2013}
Abdellah Chkifa, Albert Cohen, Ronald DeVore, and Christoph Schwab.
\newblock Sparse adaptive {T}aylor approximation algorithms for parametric and
  stochastic elliptic {PDEs}.
\newblock {\em ESAIM: Mathematical Modelling and Numerical Analysis},
  47(1):253--280, 2013.

\bibitem{CohenDevoreSchwab2011}
Albert Cohen, Ronald Devore, and Christoph Schwab.
\newblock Analytic regularity and polynomial approximation of parametric and
  stochastic elliptic {PDE}'s.
\newblock {\em Analysis and Applications}, 9(01):11--47, 2011.

\bibitem{DashtiStuart17}
Masoumeh Dashti and Andrew~M. Stuart.
\newblock The {B}ayesian approach to inverse problems.
\newblock In Roger Ghanem, David Higdon, and Houman Owhadi, editors, {\em
  Handbook of Uncertainty Quantification}, pages 311--428. Springer
  International Publishing, Cham, 2017.

\bibitem{ErnstSprungkTamellini2018}
Oliver~G Ernst, Bj\"orn Sprungk, and Lorenzo Tamellini.
\newblock Convergence of sparse collocation for functions of countably many
  gaussian random variables (with application to elliptic {PDEs}).
\newblock {\em SIAM Journal on Numerical Analysis}, 56(2):877--905, 2018.

\bibitem{GarreisUlbrich2017}
Sebastian Garreis and Michael Ulbrich.
\newblock Constrained optimization with low-rank tensors and applications to
  parametric problems with {PDEs}.
\newblock {\em SIAM Journal on Scientific Computing}, 39(1):A25--A54, 2017.

\bibitem{GerstnerGriebel2003}
Thomas Gerstner and Michael Griebel.
\newblock Dimension--adaptive tensor--product quadrature.
\newblock {\em Computing}, 71(1):65--87, 2003.

\bibitem{GlowinskiLions95}
Roland Glowinski and Jacques-Louis Lions.
\newblock Exact and approximate controllability for distributed parameter
  systems.
\newblock {\em Acta numerica}, 4:159--328, 1995.

\bibitem{GrahamKuoNicholsEtAl15}
Ivan~G Graham, Frances~Y Kuo, James~A Nichols, Robert Scheichl, Ch~Schwab, and
  Ian~H Sloan.
\newblock Quasi-{M}onte {C}arlo finite element methods for elliptic {PDEs} with
  lognormal random coefficients.
\newblock {\em Numerische Mathematik}, 131(2):329--368, 2015.

\bibitem{GunzburgerMing2011}
Max Gunzburger and Ju~Ming.
\newblock Optimal control of stochastic flow over a backward-facing step using
  reduced-order modeling.
\newblock {\em SIAM Journal on Scientific Computing}, 33(5):2641--2663, 2011.

\bibitem{Gunzburger03}
Max~D Gunzburger.
\newblock {\em {Perspectives in Flow Control and Optimization}}, volume~5.
\newblock Siam, 2003.

\bibitem{GunzburgerLeeLee11}
Max~D Gunzburger, Hyung-Chun Lee, and Jangwoon Lee.
\newblock Error estimates of stochastic optimal {N}eumann boundary control
  problems.
\newblock {\em SIAM Journal on Numerical Analysis}, 49(4):1532--1552, 2011.

\bibitem{HinzePinnauUlbrichEtAl2008}
Michael Hinze, Ren{\'e} Pinnau, Michael Ulbrich, and Stefan Ulbrich.
\newblock {\em Optimization with PDE Constraints}, volume~23.
\newblock Springer Science \& Business Media, 2008.

\bibitem{HoangSchwab14}
Viet~Ha Hoang and Christoph Schwab.
\newblock N-term {W}iener chaos approximation rates for elliptic {PDEs} with
  lognormal {G}aussian random inputs.
\newblock {\em Mathematical Models and Methods in Applied Sciences},
  24(04):797--826, 2014.

\bibitem{HouLeeManouzi2011}
L.S. Hou, J.~Lee, and H.~Manouzi.
\newblock Finite element approximations of stochastic optimal control problems
  constrained by stochastic elliptic {PDE}s.
\newblock {\em Journal of Mathematical Analysis and Applications},
  384(1):87--103, 2011.

\bibitem{Kouri2014}
Drew~P Kouri.
\newblock A multilevel stochastic collocation algorithm for optimization of
  {PDEs} with uncertain coefficients.
\newblock {\em SIAM/ASA Journal on Uncertainty Quantification}, 2(1):55--81,
  2014.

\bibitem{KouriHeinkenschloosVanBloemenWaanders2012}
Drew~P Kouri, Matthias Heinkenschloss, Denis Ridzal, and Bart~G van
  Bloemen~Waanders.
\newblock A trust-region algorithm with adaptive stochastic collocation for
  {PDE} optimization under uncertainty.
\newblock {\em SIAM Journal on Scientific Computing}, 35(4):1847--1879, 2012.

\bibitem{KouriSurowiec2016}
Drew~P Kouri and Thomas~M. Surowiec.
\newblock Risk-averse {PDE}-constrained optimization using the conditional
  {Value-At-Risk}.
\newblock {\em SIAM Journal on Optimization}, 26(1):365--396, 2016.

\bibitem{KunothSchwab2013}
Angela Kunoth and Christoph Schwab.
\newblock Analytic regularity and {GPC} approximation for control problems
  constrained by linear parametric elliptic and parabolic {PDEs}.
\newblock {\em SIAM Journal on Control and Optimization}, 51(3):2442--2471,
  2013.

\bibitem{KunothSchwab2016}
Angela Kunoth and Christoph Schwab.
\newblock Sparse adaptive tensor {G}alerkin approximations of stochastic
  {PDE}-constrained control problems.
\newblock {\em SIAM/ASA Journal on Uncertainty Quantification},
  4(1):1034--1059, 2016.

\bibitem{LassilaManzoniQuarteroniEtAl2013}
Toni Lassila, Andrea Manzoni, Alfio Quarteroni, and Gianluigi Rozza.
\newblock Boundary control and shape optimization for the robust design of
  bypass anastomoses under uncertainty.
\newblock {\em ESAIM: Mathematical Modelling and Numerical Analysis},
  47(4):1107--1131, 2013.

\bibitem{LeeGunzburger2017}
Hyung-Chun Lee and Max~D Gunzburger.
\newblock Comparison of approaches for random {PDE} optimization problems based
  on different matching functionals.
\newblock {\em Computers \& Mathematics with Applications}, 73(8):1657--1672,
  2017.

\bibitem{LeeLee2013}
Hyung-Chun Lee and Jangwoon Lee.
\newblock A stochastic {G}alerkin method for stochastic control problems.
\newblock {\em Communications in Computational Physics}, 14(1):77--106, 2013.

\bibitem{Lions1971}
Jacques~Louis Lions.
\newblock {\em Optimal Control of Systems Governed by Partial Differential
  Equations}, volume 170.
\newblock Springer Berlin, 1971.

\bibitem{NgWillcox2014}
Leo~WT Ng and Karen~E Willcox.
\newblock Multifidelity approaches for optimization under uncertainty.
\newblock {\em International Journal for Numerical Methods in Engineering},
  100(10):746--772, 2014.

\bibitem{PeherstorferWillcoxGunzburger18}
Benjamin Peherstorfer, Karen Willcox, and Max Gunzburger.
\newblock Survey of multifidelity methods in uncertainty propagation,
  inference, and optimization.
\newblock {\em SIAM Review}, 60(3):550--591, 2018.

\bibitem{Quarteroni2013}
Alfio Quarteroni.
\newblock {\em Numerical {M}odels for {D}ifferential {P}roblems}.
\newblock Springer, Milano, 2nd ed., 2013.

\bibitem{RosseelWells2012}
Eveline Rosseel and Garth~N Wells.
\newblock Optimal control with stochastic {PDE} constraints and uncertain
  controls.
\newblock {\em Computer Methods in Applied Mechanics and Engineering}, 213 -
  216(0):152 -- 167, 2012.

\bibitem{SchillingsSchmidtSchulz2011}
Claudia Schillings, Stephan Schmidt, and Volker Schulz.
\newblock Efficient shape optimization for certain and uncertain aerodynamic
  design.
\newblock {\em Computers \& Fluids}, 46(1):78--87, 2011.

\bibitem{SchillingsSchwab2013}
Claudia Schillings and Christoph Schwab.
\newblock Sparse, adaptive {S}molyak quadratures for {B}ayesian inverse
  problems.
\newblock {\em Inverse Problems}, 29(6):065011, 2013.

\bibitem{TieslerKirbyXiuEtAl2012}
Hanne Tiesler, Robert~M Kirby, Dongbin Xiu, and Tobias Preusser.
\newblock Stochastic collocation for optimal control problems with stochastic
  {PDE} constraints.
\newblock {\em SIAM Journal on Control and Optimization}, 50(5):2659--2682,
  2012.

\bibitem{Troeltzsch2010}
Fredi Tr{\"o}ltzsch.
\newblock {\em Optimal Control of Partial Differential Equations: Theory,
  Methods, and Applications}, volume 112.
\newblock American Mathematical Soc., 2010.

\bibitem{ZechSchwab2017}
Jakob Zech and Christoph Schwab.
\newblock Convergence rates of high dimensional {S}molyak quadrature.
\newblock Technical Report 2017-27, Seminar for Applied Mathematics, ETH
  Z{\"u}rich, Switzerland, 2017.

\end{thebibliography}

\end{document}